\documentclass[12pt,reqno]{amsart}

\usepackage[utf8]{inputenc}

\usepackage[normalem]{ulem}
\usepackage{ulem}
\usepackage[french,english]{babel}

\DeclareMathAlphabet{\mathpzc}{OT1}{pzc}{m}{it}

\usepackage{changepage, array}
\usepackage{xcolor}
\usepackage{graphicx}
\usepackage{comment}
\usepackage{hyperref}
\usepackage{enumerate}

\author{Hung Yean Loke}
\address{National University of Singapore, Science Drive 2, Singapore 117543}
\email{matlhy@nus.edu.sg}
\author{Tomasz Przebinda}
\address{Department of Mathematics, University of Oklahoma, Norman, OK 73019, USA}
\email{tprzebinda@ou.edu}

\title[A Cauchy - Harish-Chandra integral for a  dual  pair over a p-adic field]{A Cauchy - Harish-Chandra integral for a  dual  pair over a p-adic field, the definition and a conjecture}

\def\Chc{\mathop{\hbox{\rm Chc}}\nolimits}
\def\chc{\mathop{\hbox{\rm chc}}\nolimits}

\def\n{\mathfrak n}
\def\g{\mathfrak g}

\def\h{\mathfrak h}
\def\sp{\mathfrak {sp}}
\def\sl{\mathfrak {sl}}

\def\gl{\mathfrak {gl}}

\def\Df{\mathbb{D}}
\def\F{\mathbb{F}}
\def\Ef{\mathbb{E}}
\def\Qp{\mathbb{Q}_p}

\def\fo{\mathfrak o}
\def\cF{\mathcal{F}}

\def\Df{\Bbb{D}}

\def\C{\mathbb{C}}

\def\Or{\mathcal{O}}

\def\a{\mathfrak a}

\def\l{\mathfrak l}

\def\ss1{\mathfrak s_{\overline 1}}

\def\hs1{\mathfrak h_{\overline 1}}

\def\Pg{\mathrm{P}}

\def\inv{\mathrm{inv}}
\def\Jac{\mathrm{Jac}}

\def\supp{\mathrm{supp}}

\def\Ker{\mathrm{Ker}}

\def\G{\mathrm{G}}
\def\N{\mathrm{N}}

\def\K{\mathrm{K}}
\def\H{\mathrm{H}}

\def\L{\mathrm{L}}
\def\Bbb{\mathbb}

\def\N{\mathrm{N}}
\def\A{\mathrm{A}}
\def\H{\mathrm{H}}
\def\T{\mathrm{T}}
\def\GL{\mathrm{GL}}
\def\SL{\mathrm{SL}}

\def\Sp{\mathrm{Sp}}

\def\Og{\mathrm{O}}
\def\Ug{\mathrm{U}}

\def\rAC{\mathrm{AC}}

\def\t{\tilde}
\def\wt{\widetilde}
\newcommand{\reg}[1]{ {#1}^{\mathrm{reg}}}

\def\W{\mathsf{W}}

\def\Uv{\mathrm{U}}
\def\V{\mathsf{V}}

\def\X{\mathsf{X}}
\def\Y{\mathsf{Y}}

\def\End{\mathop{\hbox{\rm End}}\nolimits}

\def\det{\mathop{\hbox{\rm det}}\nolimits}
\def\ad{\mathop{\hbox{\rm ad}}\nolimits}
\def\Ad{\mathop{\hbox{\rm Ad}}\nolimits}
\def\Hom{\mathop{\hbox{\rm Hom}}\nolimits}

\def\Im{\mathop{\hbox{\rm Im}}\nolimits}

\def\tr{\mathop{\hbox{\rm tr}}\nolimits}

\def\WF{\mathop{\hbox{\rm WF}}\nolimits}

\def\lim{\mathop{\hbox{\rm lim}}\nolimits}

\def\Ind{\mathop{\hbox{\rm Ind}}\nolimits}

\def\ker{\mathop{\hbox{\rm ker}}\nolimits}

\def\Z{\mathcal{Z}}

\def\Ss{\mathcal{S}}
\def\Ee{\mathcal{E}}

\def\fontindex{\arabic}

\def\fonttitre{\textsf}
\newcounter{thh}

\newtheorem{thm}[thh]{\fonttitre{Theorem}}

\newtheorem{pro}[thh]{\fonttitre{Proposition}}
\newtheorem*{pro*}{\fonttitre{Proposition}}
\newtheorem{cor}[thh]{\fonttitre{Corollary}}
\newtheorem*{coro*}{\fonttitre{Corollary}}
\newtheorem{lem}[thh]{\fonttitre{Lemma}}

\newtheorem{defi}[thh]{\fonttitre{Definition}}
\newtheorem*{defi*}{\fonttitre{Definition}}
\newtheorem{conj}{\fonttitre{Conjecture}}

\newtheorem*{nota*}{\fonttitre{Notation}}

\theoremstyle{remark}
\newtheorem{rem}{\fonttitre{Remark}}

\newenvironment{prf}{\begin{proof}}{\end{proof}}

\def\muet{ \ifthenelse{\equal{a}{b}}}
\def\nn{\nonumber}
\def\Z'{\Bbb{Z}'}

\def\biblio{\sloppy
\bibliographystyle{alpha}
\bibliography{article}}

\begin{document}
\thanks{The first author is grateful to the University of Oklahoma for hospitality and financial support in February 2017. The second author gratefully acknowledges hospitality and financial support from the Institute of Mathematical Sciences at the National University of Singapore and the National Science Foundation under Grant DMS-2225892. }

\date{}
\subjclass[2010]{Primary: 22E46; secondary: 22E47} 
\keywords{Howe's correspondence, reductive dual pairs over non archimedean local fields, characters.}

\date{\today}
\maketitle
For a real irreducible dual pair there is  an integral kernel operator 
which maps the distribution character  of an irreducible admissible representation  of the group with the smaller or equal rank to an invariant 
eigendistribution  on the group with the larger or equal rank. 

The purpose of this article is to transfer this construction to the p-adic case.
We provide the precise definition of the integral kernel operator and formulate a conjecture. 
\tableofcontents

\def\v0{v_0}
\def\u0{u_0}
\def\sm{\mathrm {sum}}

\section{\bf Introduction}

For a real irreducible dual pair $(\G,\G')$ with the rank of $\G'$ less or equal to the rank of $\G$ \cite{PrzebindaCauchy} provides an integral kernel operator $\Chc$ which maps the distribution character $\Theta_{\Pi'}$ of an irreducible admissible representation $\Pi'$ of $\wt\G'$ to an invariant eigendistribution $\Theta'_{\Pi'}$ on the group $\wt\G$ with the correct infinitesimal character in the sense of the infinitesimal character correspondence,  see \cite{BerPrzeCHC_inv_eig} and \cite{PrzebindaInfinitesimal}. If the pair is in the stable range with $\G'$ the smaller member and if the representation is unitary, then $\Theta'_{\Pi'}=\Theta_\Pi$, where $\Pi$ is associated to $\Pi'$ via Howe's correspondence, \cite{PrzebindaStableUnitary}. The same happens if $\G'=\Ug_{p',q'}$ and $\G=\Ug_{p,q}$ with $p'+q'=p+q$ and $\Pi'$ is a discrete series representation, \cite{Merino2021discreteseries}. The acronym $\Chc$ stands for the Cauchy Harish-Chandra integral, because as explained in \cite{PrzebindaCauchy}, the construction gives a direct link from the Cauchy determinantal identity through Harish-Chandra's theory of the semisimple orbital integrals to Howe's correspondence.

The purpose of this article is to transfer the construction of \cite{PrzebindaCauchy} to the p-adic case.
We provide the precise definition of the integral kernel operator $\Chc$ and formulate a conjecture that roughly $\Theta'_{\Pi'}=\Theta_\Pi$. Since this takes a considerable space, we will leave reporting on some properties of the new p-adic $\Chc$ for future papers. 

We are interested in groups that occur as members of dual pairs, as defined in \cite{howetheta}. Thus we fix a non-archimedean local field $\F$ of characteristic $0$. A member of a dual pair of type II is the general linear group $\G=\GL_\Df(\V)$, where $\Df$ is a finite dimensional division algebra over $\F$ and $\V$ is a finite dimensional right $\Df$ vector space. The Lie algebra of $\GL_\Df(\V)$ is $\End_\Df(\V)$.
A member of a dual pair of type I is the isometry group $\G \subseteq \GL_\Df(\V)$ of a non-degenerate $\sigma$-hermitian or $\sigma$-skew-hermitian form $(\cdot,\cdot)$ on $\V$, where $\sigma$ is a possibly trivial involution on $\Df$, fixing $\F$ pointwise.  The Lie algebra $\g$ of $\G$ is contained in $\End_\Df(\V)$.

\section{\bf Cartan subgroups.} 
\ Below we shall describe the Cartan subalgebras in $\g$ and Cartan subgroups in $\G$. By definition, a Cartan subalgebra of $\g$ is the centralizer of a regular semisimple element, and a Cartan subgroup of $\G$ is the centralizer of a Cartan subalgebra. They occur in the literature in the cases when $\Df$ is commutative, see \cite{MorrisTamely}  and \cite{Ju-LeeKim_HeckeAlgebras}. The proofs
given there require some knowledge of the theory of algebraic groups. Our arguments circumvent that theory.

\subsection{\bf Semisimple elements.}\label{General linear groups}
\ We will  say that an element $x\in\End_\Df(\V)$ is semisimple if $\V$ decomposes into a direct sum of 
$x$-irreducible subspaces over $\Df$. 
Then $x$ is semisimple as an element of $\End_\F(\V)$. Indeed, since the field $\F$ is of characteristic zero, $x$ has the Jordan-Chevalley decomposition, $x=x_s+x_n$, where $x_s$ is semisimple and $x_n$ is nilpotent and both are polynomials in $x$ with coefficients in $\F$, see
\cite[pages 96-98]{JacobsonLie}. Hence they commute with the action of $\Df$. Therefore, $x_s, x_n\in  \End_\Df(\V)$. 
On each $x$-irreducible component $U$, $x_n$ is nilpotent so $\Ker\, x_n |_U$ is nonzero. Hence $\Ker\, x_n|_U = U$. This implies that $x_n=0$.  
Being semisimple as an element of $\End_\F(\V)$ is equivalent to the minimal polynomial $P_x(t)\in\F[t]$ being a product of irreducible polynomials, each with multiplicity $1$, see \cite[Exercise 10, page 662]{LangAlgebra2002}.  In particular the subset of the semisimple elements is Zariski dense.
\subsection{\bf Regular elements.}\label{Regular elements}
\ Given a finite dimensional vector space $\Uv$ over a commutative field  $\F$ and a linear map $\L \colon \Uv\to\Uv$, there are $\L$-invariant subspaces $\Uv_{0,\L}$ and $\Uv_{1,\L}$ such that $\L$ restricted to $ \Uv_{0,\L}$ is nilpotent, $\L$ restricted to $ \Uv_{1,\L}$ is bijective and  $\Uv=\Uv_{0,\L}\oplus\Uv_{1,\L}$. 
This is called the Fitting decomposition of $\Uv$ with respect to $\L$, see \cite[ pages 37-38]{JacobsonLie}. The subspace $\Uv_{0,\L}$ is called the Fitting null component of $\Uv$ relative to $\L$.

Let $l$ be the minimal dimension of the Fitting null component of $\ad x$ as $x$ varies through~$\g$, 
\[
l= \min\{\dim_{\F} (\g_{0,{\mathrm{ad}} \, x}) \ : x\in\g\}\,.
\]
Following Harish-Chandra, \cite[page 63]{Harish-Chandra-van-Dijk}, we introduce a polynomial function $\eta \colon \g\to\F$ by
\begin{equation}\label{functioneta}
\det(t-\ad x)_\g=\eta(x) t^l+ \ldots \text{terms of higher degree in}\ t \qquad (x\in\g)\,.
\end{equation}
According to Harish-Chandra, $x\in\g$ is regular if and only if $\eta(x)\ne 0$. 
Thus the subset of the regular elements is Zariski open, by definition.

\subsection{\bf A general linear group.}\label{A general linear group}
\ 
\begin{lem}\label{regularissemisimple0}
Let $\Ef$ be a field extension of $\F$. Suppose $v_1, v_2, \ldots, v_m\in \F^n$ are linearly independent over $\F$. Then they are also linearly independent over $\Ef$ in $\Ef^n$.
\end{lem}

\begin{prf}
We may write $v_i$ as a row vector. Then $v_1, \ldots, v_m$ form an $m$ by $n$ matrix $A$ over~$\F$. Let $A'$ be its reduced row echelon form over $\F$. It is the same reduced echelon matrix~$A'$ if we work over $\Ef$. The rank of the matrix~$A'$ is $m$ since the vectors are $\F$-linearly independent. This implies that the $m$ is also the dimension of the span of the vectors over~$\Ef$.
\end{prf}
\begin{lem}\label{regularissemisimple1}
Let $\Ef$ be a field extension of $\F$ and let $X\in M_{n,n}(\F)$. Denote by $P_\F(t)\in\F[t]$ the minimal polynomial of $X$ and by $P_\Ef(t)\in\Ef[t]$ the minimal polynomial of $X$ viewed as an element of 
$M_{n,n}(\Ef)$. Then $P_\Ef(t)=P_\F(t)$.
\end{lem}
\begin{prf}
Clearly $P_\Ef(t)$ divides $P_\F(t)$ in $\Ef[t]$. Let $d$ denote the degree of $P_\Ef(t)$. Then the vectors
\[
I\,,\ \ \ X\,,\ \ \ X^2\,,\ \ \ ...\,,\ \ \ X^d\in M_{n,n}(\F)
\]
are linearly dependent over $\Ef$. Since each of them is in $M_{n,n}(\F)$, Lemma \ref{regularissemisimple0} shows that
they
are linearly dependent over $\F$. Hence, 
\[
\deg P_\F(t)\leq \deg P_\Ef(t)\,.
\]
Therefore $P_\Ef(t)=P_\F(t)$.
\end{prf}
\begin{lem}\label{regularissemisimple}
Any regular element in $\End_\F(\V)$ is semisimple.
\end{lem}
\begin{prf}
In \cite[chapter III, section 1, pages 60-61]{JacobsonLie} Jacobson proves that any regular element is contained in a Cartan subalgebra, as defined by him, \cite[page 57]{JacobsonLie}. In \cite[chapter III, section 3, Theorem 2]{JacobsonLie}  he shows that if $\F$ is algebraically closed, then his Cartan subalgebra consists of semisimple elements, as defined in this article. However, by Lemma \ref{regularissemisimple1} the minimal polynomial remains the same after a field extension. Hence, by \cite[Exercise 10, page 662]{LangAlgebra2002} a regular element in $\End_\F(\V)$ is semisimple. 
\end{prf}
As is well known, $\End_\Df(\V)$ is a central simple algebra with the center equal to $Z(\Df)$ (times the identity), see  \cite[Theorem 5.5, page 656]{LangAlgebra2002}.
Recall, \cite[Theorem, page 236]{Pierce}, that for any central simple algebra $A$, $\dim_{Z(A)} A=(\deg A)^2$, where $\deg A$ is a positive integer called the degree of $A$. In particular,
\[
\deg \End_\Df(\V)=\dim_\Df\V\cdot \deg \Df\,.
\]
\begin{lem}\label{EndD1}
Let $x\in\End_\Df(\V)$ be such that $\V$ is $x$-irreducible over $\Df$. Then
\[
\dim_{Z(\Df)}\End_\Df(\V)^x\geq \deg \End_\Df(\V)
\]
with the equality if and only if
\[
\End_\Df(\V)^x=Z(\Df)[x]\,.
\]
Furthermore, $Z(\Df)[x]$ is a field.
\end{lem}
\begin{prf}
Schur's Lemma, \cite[p.118]{JacobsonBasicII}, implies that $\End_\Df(\V)^x$ is a division algebra over 
$Z(\Df)$. Hence, $Z(\Df)[x]\subseteq \End_\Df(\V)^x$ is a subfield, see  \cite[Lemma b, page 235]{Pierce}. In particular 
\begin{equation} \label{eqZDlessEndV}
\dim_{Z(\Df)}Z(\Df)[x] \leq  \dim_{Z(\Df)} \End_\Df(\V)^x
\end{equation}
and $Z(\Df)[x]$ is a simple algebra over $Z(\Df)$. Applying the theorem on page 232 
in \cite{Pierce} to the two algebras $Z(\Df)[x]\subseteq \End_\Df(\V)$, we see that $Z(\Df)[x]$ and $\End_\Df(\V)^x$
are mutual centralizers. In other words, $Z(\Df)[x]$ is the center of $\End_\Df(\V)^x$. Moreover
\begin{equation} \label{eqdimZDdimEndD}
\dim_{Z(\Df)}Z(\Df)[x]\cdot \dim_{Z(\Df)} \End_\Df(\V)^x=\dim_{Z(\Df)}\End_\Df(\V)\ = (\deg\End_\Df(\V))^2.
\end{equation}
Using \eqref{eqZDlessEndV}, we get the inequality
\[
\dim_{Z(\Df)}\End_\Df(\V)^x\geq \deg \End_\Df(\V).
\]
If $Z(\Df)[x]=\End_\Df(\V)^x$, then the above is an equality.

Conversely suppose $\dim_{Z(\Df)}\End_\Df(\V)^x = \deg \End_\Df(\V)$. By \eqref{eqdimZDdimEndD} we get
\[
\dim_{Z(\Df)}Z(\Df)[x] = \deg\End_\Df(\V).
\]
This implies that $Z(\Df)[x]$ is maximal (equivalently, strictly maximal) subfield  of $\End_\Df(\V)$. Thus it is also a maximal subfield of $\End_\Df(\V)^x$. Since $Z(\Df)[x]$ is the center of $\End_\Df(\V)^x$, we get $Z(\Df)[x]=\End_\Df(\V)^x$.
\end{prf}
\begin{lem}\label{EndD2}
Let $x\in\End_\Df(\V)$ be semisimple. Then
\begin{equation}\label{EndD2.1}
\dim_{Z(\Df)}\End_\Df(\V)^x\geq \deg \End_\Df(\V)
\end{equation}
with the equality if and only if
\begin{equation}\label{EndD2.2}
\V=\V_1\oplus \V_2\oplus \ldots \oplus \V_n
\end{equation}
is the direct sum of $x$-irreducibles over $\Df$ with each
\begin{equation}\label{EndD2.3}
\End_\Df(\V_j)^x=Z(\Df)[x|_{\V_j}]
\end{equation}
a field, and 
\begin{equation}\label{EndD2.4}
\End_\Df(\V)^x=\End_\Df(\V_1)^x\oplus \End_\Df(\V_2)^x\oplus \ldots \oplus \End_\Df(\V_n)^x\,.
\end{equation}
\end{lem}
\begin{prf}
Fix a decomposition \eqref{EndD2.2}. Then 
\begin{equation}\label{EndD2.5}
\End_\Df(\V)^x\supseteq \End_\Df(\V_1)^x\oplus \End_\Df(\V_2)^x\oplus \ldots \oplus \End_\Df(\V_n)^x
\end{equation}
and by Lemma \ref{EndD1}, 
\begin{equation}\label{EndD2.6}
\dim_{Z(\Df)}\End_\Df(\V_j)^x\geq \deg \End_\Df(\V_j)
\end{equation}
for all $j$. This implies \eqref{EndD2.1}. Conversely equality in \eqref{EndD2.1} implies equalities in \eqref{EndD2.6} and \eqref{EndD2.5}. 
Hence \eqref{EndD2.3} follows from Lemma \ref{EndD1}.
\end{prf}
For a regular semisimple element $x\in \End_\Df(\V)$, one can show that the decomposition \eqref{EndD2.2} and \eqref{EndD2.4} hold.
Thus the $Z(\Df)$-subalgebra generated by the restriction of $x$ to $\V_j$ in $\End_\Df(\V_j)$ is a field extension $\Ef_j$ of $\F$ and  a strictly maximal subfield  of $\End_\Df(\V_j)$. The centralizer of $x$ in $\End(\V)$ is the direct sum
\begin{equation}\label{csainglD}
\End(\V)^x=\Ef_1\oplus\Ef_2\oplus \ldots \oplus\Ef_n\,.
\end{equation}
This is a Cartan subalgebra of $\End(\V)$ and the above formula describes all of them up to conjugation.
The centralizer of a Cartan subalgebra in the group $\GL(\V)$ is called a Cartan subgroup. In the above terms
\begin{equation}\label{csginglD}
\GL_\Df(\V)^x=\Ef_1^\times\times\Ef_2^\times\times \ldots \times\Ef_n^\times\,.
\end{equation}
For a Cartan subgroup 
\[
\H = \Ef_1^\times \times \Ef_2^\times \times \ldots \times \Ef_n^\times
\] 
as above, we define the split part of $\H$ as
\[
\A=\underbrace{\F^\times \times\ldots \times \F^\times}_{n \text{ copies}}\subseteq \H\,.
\]
Then $\H/\A$ is compact. Indeed, this follows from the lemma below, which we write in terms of \cite[section I.4]{Weil_Basic}. Denote by $| \cdot |_\Df$ the module on $\Df$, as in \cite[page 4]{Weil_Basic}.

\begin{lem}\label{compactnessOfQuotient}
Denote by $\fo'\subseteq \Df$ and $\fo\subseteq \F$ the corresponding rings of integers. Let $\varpi'\in \Df$ and $\varpi\in \F$ be prime elements and let $e$ denote the order of ramification of $\Df$ over $\F$. Then
\[
\Df^\times=\sum_{j=0}^{e-1}\varpi'{}^j \fo'{}^\times \F^\times\,.
\]
\end{lem}

\begin{prf}
Since $\Df^\times=\langle \varpi'\rangle \fo'^\times$, where $\langle \varpi'\rangle=\{\varpi'{}^n;\ n\in \Bbb Z\}$, see \cite[page 32]{Weil_Basic}, it suffices to check that for any $n\in \Bbb Z$ there are $k\in \Bbb Z$, $0\leq j\leq e-1$ and $x\in \fo'^\times$ such that
\[
\varpi'{}^n=\varpi'{}^j x\varpi^k \,.
\]
Equivalently
\[
|\varpi'{}^{(n-j)}\varpi^{-k})|_{\Df}= 1\,.
\]
Let $q'$ be the modulus of $\Df$. Then
\[
|\varpi'{}^{(n-j)}\varpi^{-k}|_{\Df}=|\varpi'{}^{(n-j)}|_{\Df}|\varpi^{-k}|_{\Df}=q'{}^{j-n}q'{}^{ek}\,.
\]
Thus we want
\[
n= ek+j\,,
\]
which is possible by taking $k$ to be the quotient and $j$ to be the remainder from the division of $n$ by $e$. 
\end{prf}

For future reference, the Lie algebras of  $\H$ and $\A$ are respectively
\begin{equation}\label{CSAin II}
\h=\Ef_1\oplus \Ef_2\oplus \ldots \oplus \Ef_n\,,\ \ \ \a=\F\oplus \F \oplus \ldots \oplus \F\,.
\end{equation}

\subsection{\bf An isometry group.}\label{An isometry group}
\ 
Let $\Df$ be a division $\F$-algebra with an involution $\sigma$.
We have the inclusions
\begin{equation}\label{theinclusions}
\F\subseteq Z(\Df)^\sigma\subseteq Z(\Df)
\subseteq \Df\,.
\end{equation}
According to \cite[page 728]{howetheta}, (see also Paul Garrett's ``Algebras with involution" online, or \cite[Theorem 2.2, page 353]{Scharlau} for a proof) one of the following three cases happens.
\begin{itemize}
\item $\Df=Z(\Df)$  and $\sigma=1$;\\
\item $\Df=Z(\Df)$  and $\sigma\ne 1$;\\
\item $\Df$ is a quaternion division algebra over $Z(\Df)$, and $\sigma$ is trivial on $Z(\Df)$.
\end{itemize}
Let $\V$ be a finite dimensional right $\Df$-vector space with a non-degenerate $\Df$-hermitian or skew-hermitian form $(\cdot,\cdot)$. In the third case above,  non-degenerate $\Df$-hermitian form is determined by its rank, see \cite[page264]{LewisD_IsometryClassification}. There are four types of non-degenerate $\Df$-skew-hermitian forms, see \cite[Theorem 3]{TsukamotoAnti-Hermitian1961}. 
In any case the form $(\cdot,\cdot)$ determines an involution on the algebra $\End_\Df(\V)$,
\[
(xu,v)=(u,\sigma(x) v)) \qquad (u,v\in\V, x\in \End_\Df(\V))\,,
\]
which coincides with the original involution on $\Df$ (times the identity acting on the left).
Recall the isometry group
\begin{equation} \label{eqisometrygroup}
\G=\{x\in \End_\Df(\V);\ x\sigma(x)=1\}=\End_\Df(\V)_{\sigma=\inv}\,, 
\end{equation}
where $\inv(x)=x^{-1}$, with the Lie algebra
\[
\g=\{x\in \End_\Df(\V);\ x+\sigma(x)=0\}=\End_\Df(\V)_{\sigma=-1}\,.
\]
Fix a non-zero semisimple element $x\in \g$.
\begin{lem}\label{gD1}
Let $x\in\g$ be a non-zero semisimple element such that $\V\ne 0$ is $x$-irreducible over $\Df$. Then
\[
\dim_{Z(\Df)^\sigma}\g^x\geq \frac{1}{2}\cdot \dim_{Z(\Df)^\sigma} Z(\Df)\cdot \deg \End_\Df(\V).
\]
Equality holds if and only if
\[
\g^x=Z(\Df)[x]_{\sigma=-1}\,.
\]
Furthermore, if equality holds, then $Z(\Df)[x]$ is a field.
\end{lem}
\begin{prf}
Since $\sigma(x)=-x$, both $Z(\Df)[x]$ and $\End_\Df(\V)^x$ are $\sigma$-invariant.
Since $\Ker(x)=0$, multiplication by $x$ is an $Z(\Df)^\sigma$-linear isomorphism between $\End_\Df(\V)^x_{\sigma=-1}$  and $\End_\Df(\V)^x_{\sigma=1}$ and similarly between $Z(\Df)[x]_{\sigma=-1}$  and $Z(\Df)[x]_{\sigma=1}$. 
Hence
\[
\dim_{Z(\Df)^\sigma}Z(\Df)[x]_{\sigma=-1}=\frac{1}{2}\dim_{Z(\Df)^\sigma}Z(\Df)[x]
\]
and
\[
\dim_{Z(\Df)^\sigma}\End_\Df(\V)^x_{\sigma=-1}=\frac{1}{2}\dim_{Z(\Df)^\sigma} \End_\Df(\V)^x  =  \frac{1}{2}\dim_{Z(\Df)^\sigma} Z(\Df) \cdot \dim_{Z(\Df)} \End_\Df(\V)^x \,.
\]
By Lemma \ref{EndD1}
\begin{equation} \label{eqlem2}
\dim_{Z(\Df)}\End_\Df(\V)^x \geq \deg \End_\Df(\V).
\end{equation}
Since $\g^x= \End_\Df(\V)^x_{\sigma=-1}$,
this gives the inequality in the lemma. Equality in the lemma holds if and only if \eqref{eqlem2} is an equality. 
By Lemma \ref{EndD1} again this happens if and only if
\[
\End_\Df(\V)^x=Z(\Df)[x].
\]
In addition $Z(\Df)[x]$ is a field.
By the above computation, this is equivalent to
\[
\g^x =  \End_\Df(\V)^x_{\sigma=-1}=Z(\Df)[x]_{\sigma=-1}\,,
\]
\end{prf}
\begin{lem}\label{gD2}
Let $x\in\g$ be semisimple as an element of $\End_\Df(\V)$. Then there is a direct sum decomposition
\begin{equation}\label{gD2.1}
\V=\bigoplus_{j\in J}(\V_j\oplus \V_{j'}) \, \oplus \, \bigoplus_{l\in L}\V_l \, \oplus \V_0\,,
\end{equation}
where $0\notin J\cup L$, 
\begin{enumerate}[(i)]
\item each  subspace $\V_j$, $\V_{j'}$, $\V_l$ is $x$-irreducible over $\Df$, 
\item the restrictions of the form $(\cdot,\cdot)$ to $\V_j$ and to $\V_{j'}$ are zero, but the restriction to $\V_j \oplus \V_{j'}$ is non-degenerate,
\item the restriction to each $\V_l$ is non-degenerate and
\item the spaces $\V_j\oplus\V_{j'}$, $\V_l$, $\V_0$ are mutually orthogonal.
\item The space $\V_0 \neq 0$ if and only if $\Df$ is commutative, the involution $\sigma$ is trivial, the form $(\cdot, \cdot)$ is symmetric and the dimension of $\V$ over $\Df$ is odd. In this case $\dim_\Df\V_0=1$ and $x|_{\V_0} = 0$.
\end{enumerate}
\end{lem}
\begin{prf}
We pick a decomposition of $\V$ into  $x$-irreducibles, group together resulting isotropic subspaces which are paired via the form, and index them by a set $J$. The remaining subspaces have the property that the restriction of the form to any of them is non-degenerate. They are indexed by a set $L$, and eventually $0$. 
\end{prf}
\begin{lem}\label{gD3}
Let $x\in\g$ be semisimple as an element of $\End_\Df(\V)$. Then, in terms of \eqref{gD2.1},
\begin{equation}\label{gD3.1}
\dim_{Z(\Df)^\sigma} \g^x\geq \frac{1}{2}\cdot \dim_{Z(\Df)^\sigma}{Z(\Df)}\cdot\left(\deg\End_\Df(\V)-\dim_\Df \V_0\right)\,.
\end{equation}
The equality occurs if and only if, in terms of \eqref{gD2.1},
\begin{equation}\label{gD3.2}
\g^x=\bigoplus_{j\in J}{Z(\Df)}[x|_{\V_j+\V_{j'}}]_{\sigma=-1} 
\oplus \bigoplus_{l\in L} {Z(\Df)}[x|_{\V_l}]_{\sigma=-1}\,.
\end{equation}
Each ${Z(\Df)}[x|_{\V_j+\V_{j'}}]$ and each ${Z(\Df)}[x|_{\V_l}]$ is a field extension of ${Z(\Df)^\sigma}$. Furthermore, each ${Z(\Df)}[x|_{\V_j+\V_{j'}}]_{\sigma=-1}$ is isomorphic to ${Z(\Df)^\sigma}[x|_{\V_j}]$ by restriction to $\V_j$.
\end{lem}

\begin{prf} Recall that $\deg \End_\Df(\V) = \dim_{\Df} \V$.
Since
\[
\g^x=\g(\V)^x\supseteq\bigoplus_{j\in J}\g(\V_j\oplus \V_{j'})^x\oplus \bigoplus_{l\in L}\g(\V_l)^x
\]
and 
\[
\g(\V_j\oplus \V_{j'})^x\supseteq \End_\Df(\V_j)^x
\] 
by restriction to $\V_j$ the claim follows from Lemmas \ref{EndD1} and \ref{gD1}.
(The fraction $\frac{1}{2}$ in \eqref{gD3.1} is necessary because we are dealing with the space $\V_j \oplus \V_{j'}$.) 
\end{prf}
The centralizer of a Cartan subalgebra in the group $\G$ is called a Cartan subgroup $\H$. In the above terms
\begin{equation}\label{csginGlD}
\H=\G^x=\prod_{j\in J}{Z(\Df)}[x|_{\V_j+\V_{j'}}]^\times_{\sigma=\inv} \times \prod_{l\in L} {Z(\Df)}[x|_{\V_l}]^\times_{\sigma=\inv}\,.
\end{equation}
Set 
\[
\X=\sum_{j\in J}\V_j\,,\ \ \ \Y=\sum_{j\in J}\V_j'\,,\ \ \ \V^{ell}=\sum_{l\in L}\V_l\,.
\]
Then we have a direct sum decomposition preserved by $\H$,
\begin{equation}\label{direct sum decomposition preserved by}
\V=\X\oplus \V^{ell}\oplus \Y\,,
\end{equation}
such that $\X$, $\Y$ are complementary isotropic subspaces, $\V^{ell}\subseteq \V$ is the orthogonal complement of $\X+\Y$ and
$\H|_{\V^{ell}}$ is a compact Cartan subgroup in the isometry group of the restriction of the form $(\cdot,\cdot)$ to $\V^{ell}$. 
In particular, $\H$ does not preserve any isotropic subspace of $\V^{ell}$. 
Also, $\H|_{\X}\subseteq \GL(\X)$ is a Cartan subgroup which determines $\H|_{\X+\Y}$ via the fact that $\Y$ may be viewed as the linear dual of $\X$.
We shall refer to the split part of $\H|_{\X+\Y}$ as to the split part of $\H$ and denote it by $\A$. Then $\H/\A$ is compact. 

\subsection{\bf The Weyl Harish-Chandra integration formula.}\label{The Weyl Harish-Chandra integration formula}
\ One important conclusion from the description of Cartan subgroups in sections \ref{A general linear group} and \ref{An isometry group} is that the number of conjugacy classes of them is finite. This is stated without any proof or reference in \cite[page 87]{Harish-Chandra-van-Dijk}. For semisimple groups, a proof is given in \cite[Sect. 4.2]{MorrisRational}. We will provide a proof below for $\GL_n(\F)$ and for certain classical groups.

Suppose $\G = \GL_n(\F)$. Since the number of isomorphic classes of unramified extensions of $\F$ of a fixed degree $n$ is finite, \cite[Proposition 7.50]{milneANT}, and the number of totally ramified extensions is also finite, \cite{Serre_formule_de_masse}, the number of conjugacy classes of Cartan subgroups in $\GL_n(\F)$ is finite. A complete argument is in \cite{KrasnerFinite}.

Next we suppose $\Df$ is a commutative field and $\G$ is the isometry group on the vector space $\V$ over $\Df$ as defined in \eqref{eqisometrygroup}. In particular $\G = \GL(\V)_{\sigma=\inv}$.  We could show that the number of conjugacy classes of Cartan subgroups is finite in the following way: Let $\T$ be a Cartan subgroup in $\G$. Then $C = Z_{\GL(\V)}(\T)$ is a Cartan subgroup of $\GL(\V)$. In particular $C = \prod E_i^\times$ where $E_i$ is a  field extension of $\Df$. Moreover $\T = C_{\sigma=\inv}$. This shows that up to isomorphism over $\F$, there are only finitely many Cartan subgroups in $\G$. 
Suppose $\T'$ be another Cartan which is isomorphic to $\T$. Then $C' = Z_{\GL(\V)}(\T')$ is isomorphic to $C$. We could define Hermitian forms on $C$ and $C'$. By adjusting the Hermitian forms we have $\V \simeq C \simeq C'$ as hermitian spaces. By Witt's theorem there exists $g \in \G$ such that $gC g^{-1} = C'$. By taking $\sigma$ invariants, we get $g \T g^{-1} = \T'$. This proves that there are finitely many conjugacy classes of Cartan subgroups in $\G$.

\medskip

The set of the $x$ in $\g$ where $\eta(x)\ne 0$ is Zariski dense. Since $\g$ is Zariski connected, the complement is Zariski closed. Hence, by \cite[(2.5.3) Proposition (i)]{Margulis_book_Discrete_subgroups}, it does not contain any non-empty subset which is open in the p-adic topology. Therefore the set of the regular semisimple elements is dense  in the p-adic topology in $\g$. Hence the union of the conjugacy classes passing through the Cartan subalgebras of $\g$ is dense in $\g$.

Recall the Weyl - Harish-Chandra integration formula on the Lie algebra $\g$,
\begin{equation}\label{Weyl - Harish-Chandra integration formula on g'}
\int_{\g}\psi(x)\,dx=\sum_{\H}\frac{1}{|W(\H)|}\int_{\reg{\h}}|\eta(x)|\int_{\G/\H}\psi(x^{g})\,d(g\H)\,dx \qquad (\psi\in \Ss(\g))\,.
\end{equation}
Here the summation is over a maximal family of mutually non-conjugate Cartan subgroups $\H\subseteq \G$, $W(\H)$ is the Weyl group equal to the quotient of the normalizer of $\H$ in $\G$ by $\H$, $dx=\,d(g\H)\,dx$, $x^g=gxg^{-1}$
see  \cite[pages 86-88]{Harish-Chandra-van-Dijk}.

Let $\G^0\subseteq \G$ denote the Zariski identity component and let $Z\subseteq \G$ denote the center of $\G$. The Cayley transform $x\mapsto (1+x)(1-x)^{-1}$ is a birational isomorphism between the Lie algebra $\g$ and the group $\G^0$. Hence the set of the regular semisimple elements in $\G^0$ is Zariski open. If $\G^0$ is Zariski connected then this set is also dense. Hence, by \cite[(2.5.3) Proposition (i)]{Margulis_book_Discrete_subgroups}, it is dense in the p-adic topology. Therefore the union of the all conjugacy classes of Cartan subgroups is dense in $Z\G^0$. (The Cayley transform is not necessary if $\G$ is a general linear group.) Thus we have 
the Weyl - Harish-Chandra integration formula on the group $Z\G^0$,
\begin{equation}\label{Weyl - Harish-Chandra integration formula on G}
\int_{\G}\Psi(x)\,dx=\sum_{\H}\frac{1}{|W(\H)|}\int_{\reg{\H}}|D(x)|\int_{\G/\H}\Psi(x^{g})\,d(g\H)\,dx \qquad (\Psi\in C_c(Z\G^0))\,.
\end{equation}
Here $D(x)$ is the coefficient of $t^l$ in the polynomial (cf. \eqref{functioneta})
\[
\det(t+1-\Ad(x))
\]
in the indeterminate $t$. From the classification of dual pairs we know that $Z\G^0=\G$ unless $\G$ is an even orthogonal group.

\subsection{\bf Some geometry of moment maps for dual pairs of type II} \label{Some geometry of moment maps for irreducible dual pairs of type II}
\  
Let $\V$ be a finite dimensional right $\Df$-vector space. 
For $x\in \End_\Df(\V)$ let $\tr(x)$ be the trace of $x$ viewed as an element of $\End_\F(\V)$. Then
\begin{equation}\label{regulartrace}
\tr\in \Hom_\F(\End_\Df(\V),\F)
\end{equation}
and in terms of \cite[page 53]{Weil_Basic},
\begin{equation}\label{oregulartraceII}
\tr_{{\mathrm{End}}_\F(\V)/\F}(x)=\dim_\Df(\V)\cdot \tr(x) \qquad (x\in \End_\Df(\V))\,.
\end{equation}
Since the trace of the identity map is not zero, we see that $\tr\ne 0$. (Here we need the assumption that the characteristic of $\F$ is zero.)
For $x, y \in \End_\Df(\V)$, define $t(x)(y) = \tr(xy) \in \F$. This gives a  nonzero map
\begin{equation}\label{identification_with_the_dual_II}
\End_\Df(\V) \rightarrow \Hom_\F(\End_\Df(\V),\F)\,,\ \ x \mapsto t(x)(-).
\end{equation}
The kernel is a two-sided ideal. The algebra $\End_\Df(\V)$ is simple so this ideal is zero. Thus \eqref{identification_with_the_dual_II} is an injection. It is a linear bijection by dimension count.

Let $\V'$ be another finite dimensional right $\Df$-vector space. We define a map
\begin{equation}\label{anotheridentification_with_the_dual_II}
T \colon \Hom_\Df(\V,\V') \rightarrow \Hom_\F(\Hom_\Df(\V',\V),\F)
\end{equation}
by
\[
T(y)(x) = \tr(xy)
\]
where $y \in \Hom_\Df(\V,\V')$ and $x \in \Hom_\Df(\V',\V)$.
Since $\Hom_\Df(\V',\V)$ is a simple $\End_\Df(\V)\times\End_\Df(\V')$-module, $\ker T = 0$. Let
\begin{equation}\label{symplectic_space_II}
\W=\Hom_\Df(\V',\V)\oplus \Hom_\Df(\V,\V')\,.
\end{equation}
The fact that $\ker T = 0$ implies that
\begin{equation}\label{symplectic_form_II}
\langle w,w'\rangle=\tr(xy')-\tr(x'y)\qquad(w=(x,y),\ w=(x',y')\in \W)
\end{equation}
defines a non-degenerate symplectic form on $\W$.

For any $\F$-subspace $\mathfrak e\subseteq \sp(\W)$, we set $\mathfrak e^*=\Hom(\mathfrak e, \F)$. We recall the unnormalized moment map
\begin{eqnarray}\label{moment_map_II}
&&\tau_{\mathfrak e^*}:\W\to \mathfrak e^*\,,\\
&&\tau_{\mathfrak e^*}(w)(e)=\langle e(w),w\rangle \qquad (e\in \mathfrak e, w\in \W)\,.\nn
\end{eqnarray}
Let $\g=\End_\Df(\V)$ and let $\g'=\End_\Df(\V')$. Define
\begin{eqnarray}\label{moment_maps_II}
&&\tau_\g \colon \W\to \g\,,\ \ \ \tau_{\g'} \colon \W\to \g'\,,\\
&&\tau_\g(x,y)=xy\,,\ \ \ \tau_{\g'}(x,y)=yx \qquad (x\in \Hom_\Df(\V',\V),\ y\in \Hom_\Df(\V,\V'))\,.\nn
\end{eqnarray}
Then under the identification \eqref{identification_with_the_dual_II} we have 
\begin{equation}\label{identification_of_moment_maps_II}
\tau_{\g^*}=2\tau_\g\,,\ \ \ \tau_{\g'{}^*}=2\tau_{\g'}\,.
\end{equation}
The groups $\G=\GL_\Df(\V)$ and $\G'=\GL_\Df(\V')$ act on $\W$ by pre-multiplication and post-multiplication by the inverse. These actions preserve the symplectic form. 
The moment maps \eqref{identification_of_moment_maps_II} intertwine these actions with the corresponding co-adjoint and adjoint actions. 

Let $\H' \subseteq \G' = \GL_\Df(\V')$ be a Cartan subgroup with the $\F$-split component $\A'\subseteq \H'$. Let
\begin{eqnarray}\label{A'isotypicdecomposition}
\V'=\V_1'\oplus\V'_2\oplus \ldots
\end{eqnarray}
be the decomposition of $\V'$ into $\A'$-isotypic components. 
Then the symplectic space  
decomposes into a direct sum of mutually orthogonal subspaces:
\begin{equation}\label{A'isotypicdecompositionW}
\W=\W_1\oplus \W_2\oplus \ldots \,,
\end{equation}
where $\W_j=\Hom_\Df(\V'_j,\V)\oplus\Hom_\Df(\V,\V'_j)$. 

Let $\A''\subseteq \Sp(\W)$ denote the centralizer of $\A'$.
Then $\A''$ preserves the decomposition \eqref{A'isotypicdecompositionW} and the restrictions to each maximal isotropic subspace $\Hom_\Df(\V'_j,\V)$ yield the following isomorphisms:
\begin{eqnarray}\label{A''a''}
\a''&=&\End_\F(\Hom_\Df(\V'_1,\V))\oplus \End_\F(\Hom_\Df(\V'_2,\V))\oplus \ldots\\
\A''&=&\GL_\F(\Hom_\Df(\V'_1,\V))\times \GL_\F(\Hom_\Df(\V'_2,\V))\times \ldots \,.\nn
\end{eqnarray}
Thus the centralizer $\A'''$ of $\A''$ in $\Sp(\W)$ is equal to $\A'$ and, as a reductive dual pair, $(\A'', \A''') = (\A'', \A')$ is isomorphic to
\begin{eqnarray}\label{A''A'''}
(\GL_{n_1}(\F), \GL_1(\F))\times (\GL_{n_2}(\F), \GL_1(\F))\times \ldots \,,
\end{eqnarray}
where $n_j=\dim_\F\Hom_\Df(\V'_j,\V)$. 

Let $\h'$ be the Lie algebra of $\H'$.  
We see from \eqref{csginglD} that $\H'\subseteq \h'$ and from~\eqref{CSAin II} that $\h'\subseteq \End_\Df(\V')$ is also an associative subalgebra over $Z(\Df)$. Denote by $\reg{\h'}\subseteq \h'$ and by $\reg{\H'}\subseteq \H'$ the subsets of the regular elements.
\begin{lem}\label{regular elements generate}
Let $x'\in \reg{\h'}$ or $x'\in \reg{\H'}$. 
Then the elements $1$,  $x'$, $x'{}^2$, $x'{}^3$, ... span the associative algebra $\h'$ over $Z(\Df)$.
\end{lem}
\begin{prf} This is immediate from \eqref{csainglD}.
\end{prf}
We shall identify the tangent bundle $T\a''$ with $\a''\times\a''$ and the cotangent bundle $T^*\a''$ with $\a''\times\a''{}^*$ as usual. Given $x'\in \h'$ there is an embedding
\begin{equation}\label{geometricII.2}
\iota_{\g, x'}:\g\ni x\to x+x'\in\a''\,.
\end{equation}
The conormal bundle to this embedding is equal to
\begin{equation}\label{geometricII.2'}
N_{\iota_{\g, x'}}=\{(x+x',\xi);\ x\in\g\,,\ \xi\in \a''{}^*\,,\ \xi|_\g=0\}
\end{equation}
Similarly, given $h'\in \H'$ there is an embedding
\begin{equation}\label{geometricII.2_1}
\iota_{\G,  h'}:\G\ni g\to gh'\in\A''\,.
\end{equation}
The conormal bundle to this embedding is equal to
\begin{equation}\label{geometricII.2'_1}
N_{\iota_{\G, h'}}=\{(gh',\xi);\ g\in\G\,,\ \xi\in \a''{}^*\,,\ \xi|_\g=0\}
\end{equation}
\begin{lem}\label{geometricII}
Let
\begin{eqnarray}\label{geometricII.1}
S_{\a''}&=&\{(x,\tau_{\a''{}^*}(w))\in \a''\times(\a''{}^*\setminus \{0\});\ x(w)=0\,,\ w\in \W\}\,,\\
S_{\A''}&=&\{(g,\tau_{\a''{}^*}(w))\in \A''\times(\a''{}^*\setminus \{0\});\ g(w)=-w\,,\ w\in \W\}\,.\nn
\end{eqnarray}
Fix an element $x'\in\reg{\h'}$ and $h'\in\reg{\H'}$. Then $S_{\a''}\cap N_{\iota_{\g, x'}}=\emptyset$
and $S_{\A''}\cap N_{\iota_{\G, h'}}=\emptyset$.
\end{lem}
\begin{prf}
Suppose $(x+x', \tau_{\a''{}^*}(w))\in  N_{\iota_{\g, x'}}$. Then the restriction $\tau_{\a''{}*}(w))|_{\g}=0$. By the definition \eqref{moment_map_II}, 
$\tau_{\a''{}*}(w))|_{\g}=\tau_{\g^*}(w)$. In terms of the decomposition \eqref{A'isotypicdecompositionW} 
\[
w=(w_1, w_2, \ldots)\,,\ \ w_j=(x_j,y_j)\,,\ \ x_j\in\Hom(\V'_j,\V)\,,\ \ y_j\in\Hom(\V,\V'_j)
\]
and in
terms of the decomposition \eqref{symplectic_space_II} 
\[
w= (\tilde{x}, \tilde{y})\,,\ \ \tilde{x} = (x_1, x_2, \ldots)\,,\ \ \tilde{y} = (y_1, y_2, \ldots)\,.
\]
Hence, by \eqref{identification_of_moment_maps_II}, 
\[
0=\tau_{\g^*}(w)(z)=2\tr(z\sum_j x_jy_j) 
\]
for all $z\in \g = \gl(\V))$. This gives
\begin{equation} \label{eqx1y1}
x_1y_1+x_2y_2+ x_3 y_3 + \ldots =0\,.
\end{equation}
We also have the condition
\begin{equation} \label{eqx1y1.0}
(x+x')w =0 \,,
\end{equation}
which means that for all $j$
\begin{equation}\label{(7.7)}
xx_j=x_jx'\,,\ \ \ y_jx=x'y_j\,.
\end{equation}
Multiplying $x^k$ from the left to \eqref{eqx1y1} and using \eqref{(7.7)} gives
\begin{equation}\label{(7.7)0}
x_1x'{}^ky_1+x_2x'{}^ky_2+\ldots=0 \qquad (k=0, 1, 2, \ldots)\,.
\end{equation}
Hence, Lemma \ref{regular elements generate} implies that 
\[
x_1z'y_1+x_2z'y_2+\ldots=0 \qquad (z'\in\h')\,.
\]
Therefore
\begin{equation}\label{(7.7)1}
x_jy_j=0
\end{equation}
for all $j$, which, by \eqref{A''a''} and \eqref{identification_of_moment_maps_II},  means that 
$\tau_{\a''^*}(w)=0$. Hence, $S_{\a''}\cap N_{\iota_{\g, x'}}=\emptyset$. 

In order to prove $S_{\A''}\cap N_{\iota_{\G, x'}}=\emptyset$, we consider an element $(gh', \tau_{\a''{}^*}(w))\in  N_{\iota_{\G, h'}}$
and replace \eqref{eqx1y1.0} by
\begin{equation} \label{eqx1y1.1}
gh'(w) =-w \,,
\end{equation}
which means that for all $j$
\begin{equation}\label{(7.7).1}
gx_j=-x_jh'\,,\ \ \ y_jg=-h'y_j\,.
\end{equation}
Multiplying $g^k$ from the left to \eqref{eqx1y1.1} and using \eqref{(7.7).1} gives 
\[
x_1h'{}^ky_1+x_2h'{}^ky_2+\ldots=0 \qquad (k=0, 1, 2, \ldots)\,.
\]
Lemma \ref{regular elements generate} implies \eqref{(7.7)} and therefore \eqref{(7.7)1}
for all $j$, which means that 
$\tau_{\a''^*}(w)=0$. Hence, $S_{\A''}\cap N_{\iota_{\G, x'}}=\emptyset$.
\end{prf}

\subsection{\bf Some geometry of moment maps for dual pairs of type I} \label{Some geometry of moment maps for irreducible dual pairs of type I}
\ 
Let $\sigma$ be an involution on the division $\F$-algebra $\Df$ fixing $\F$ pointwise.
Let $\V$ and $\V'$ be two finite dimensional right $\Df$-vector spaces with non-degenerate forms 
$(\cdot,\cdot)$ and $(\cdot,\cdot)'$ respectively, one $\sigma$-hermitian and the other one $\sigma$-skew-hermitian. 
Denote by $\G\subseteq \End_\Df(\V)$ and $\G'\subseteq \End_\Df(\V')$ the corresponding isometry groups and, by $\g\subseteq \End_\Df(\V)$ and $\g'\subseteq \End_\Df(\V')$ their Lie algebras, respectively.

Define a map
\begin{eqnarray}\label{Some geometry of moment maps for irreducible dual pairs of type I.1}
&&\Hom_\Df(\V',\V)\ni w\to w^*\in \Hom_\Df(\V,\V')\,,\\
&&(w(v'),v)=(v',w^*(v))' \qquad (v\in \V, v'\in \V')\,.\nn
\end{eqnarray}
Set $\W=\Hom_\Df(\V',\V)$. 

The formula 
\begin{equation}\label{Some geometry of moment maps for irreducible dual pairs of type I.2}
\langle w', w\rangle=\tr(w'w^*) \qquad (w,w'\in\W)
\end{equation}
defines a non-degenerate symplectic form on the $\F$-vector space $\W$. 
Define
\begin{eqnarray}\label{moment_maps_I}
&&\tau_\g \colon \W\to \g\,,\ \ \ \tau_{\g'} \colon \W\to \g'\,,\\
&&\tau_\g(w)=ww^*\,,\ \ \ \tau_{\g'}(w)=w^*w \qquad (w\in \W)\,.\nn
\end{eqnarray}
Since the $\F$-valued bilinear form
\[
\Df\times\Df\ni (a,b) \to \tr_{\Df/\F}(ab)\in \F
\]
is non-degenerate, the $\F$-valued bilinear form
\begin{equation}\label{Fvaluedbilinearform_I}
\tr_{\Df/\F}\circ(\cdot,\cdot)
\end{equation}
equal to the composition of $\tr_{\Df/\F}$ with the form  $(\cdot,\cdot)$ on $\V$ is non-degenerate. 
From the classification of the division algebras with involution, \cite[Theorem 2.2, page 353]{Scharlau}, and from the assumption that $\F\subseteq Z(\Df)^\sigma$ we see that
\begin{equation}\label{sigmatraceD_I}
\tr_{\Df/\F}(a)=\tr_{\Df/\F}(\sigma(a)) \qquad (a\in \Df)\,.
\end{equation}
Since $\tr$ may be computed in terms of a basis and a dual basis of $\F$ viewed as a vector space over $\F$, \eqref{sigmatraceD_I} implies
\begin{equation}\label{trsigaistrace_I}
\tr(x)=\tr(\sigma(x))\qquad (x\in \End_\Df(\V))\,.
\end{equation}
This in turn implies that the two subspaces
\[
\End_\Df(\V)_{\sigma=1} \ \text{ and }  \ \End_\Df(\V)_{\sigma=-1} 
\]
are orthogonal in $\End_\Df(\V)$ with respect to the pairing
\[
\End_\Df(\V)\times \End_\Df(\V)\ni (x,y)\to\tr(xy)\in\F\,.
\]
Since $\g=\End_\Df(\V)_{\sigma=-1}$
the identification \eqref{identification_with_the_dual_II} restricts to  identification $\g=\g^*$ and similarly $\g'=\g'{}^*$. Recall the moment map \eqref{moment_map_II}. 
In these terms we have
\begin{equation}\label{identification_of_moment_maps_I}
\tau_{\g^*}=\tau_\g\,,\ \ \ \tau_{\g'{}^*}=\tau_{\g'}\,.
\end{equation}
Let $\H' \subseteq \G'$ be a Cartan subgroup with the split component $\A'\subseteq \H'$. 
Recall the decomposition \eqref{direct sum decomposition preserved by}
\[
\V'=\X'\oplus\V'{}^{ell}\oplus \Y'\,.
\]
Let
\begin{equation}\label{A'isotypicdecompositionI}
\X'=\X_1'\oplus\X'_2\oplus \ldots \,,\ \ \ \Y'=\Y_1'\oplus\Y'_2\oplus \ldots \,,
\end{equation}
be the decompositions into $\A' $-isotypic components so that $\Y'_j$ is dual to $\X'_j$ via the symplectic form $\langle\cdot,\cdot\rangle$. Also let
\[
\W^{split}=\Hom(\X'\oplus\Y', \V)\,,\ \ \ \W^{ell}=\Hom(\V'{}^{ell}, \V)\,.
\]
Let 
\[
\W_j=\Hom(\X'_j\oplus \Y_j', \V)=\Hom(\X'_j, \V)\oplus \Hom(\Y_j',\V)\,.
\]
We have 
\[
\W^{split} = \Hom(\X', \V)\oplus \Hom(\Y', \V) = \bigoplus_{j} \W_j.
\]
The symplectic space $\W$ decomposes into a direct sum of mutually orthogonal subspaces:
\begin{equation}\label{A'isotypicdecompositionWI}
\W = \W^{ell}\oplus\W^{split} = \W^{ell}\oplus \left(\bigoplus_{j} \W_j \right).
\end{equation}
Furthermore,
\begin{equation}\label{A'isotypicdecompositionWIpolarization}
\W^{split}= \Hom(\X', \V)\oplus \Hom(\Y', \V)
\end{equation}
is a complete polarization.
The group $\A''$ preserves the decomposition \eqref{A'isotypicdecompositionWI} and the obvious restrictions yield the following isomorphisms:
\begin{eqnarray}\label{A''a''I.1}
\a''&=&\sp(\W^{ell})\oplus \End_\F(\Hom(\X'_1, \V))\oplus \End_\F(\Hom(\X'_2, \V))\oplus\\
\A''&=&\Sp(\W^{ell})\times \GL_\F(\Hom(\X'_1, \V))\times \GL_\F(\Hom(\X'_2, \V))\times\ldots \,.\nn
\end{eqnarray}
Thus $\A'''=\A'$ and, as a reductive dual pair, $(\A'', \A''') = (\A'', \A')$ is isomorphic to
\begin{eqnarray}\label{A''A'''I.2}
(\Sp_{2n_0}(\F), \Og_1(\F))\times(\GL_{n_1}(\F), \GL_1(\F))\times (\GL_{n_2}(\F), \GL_1(\F))\times\ldots \,,
\end{eqnarray}
where $2n_0=\dim_\F\W^{ell}$ and for $1\leq j$, $n_j=\dim_\F\Hom(\V'_j,\V))$. 
\begin{lem}\label{regular elements generate I}
Let $\h'$ be the Lie algebra  of $\H'$ and 
let $x'\in \reg{\h'}$  or $x'\in \reg{\H'}$. 
Then, in terms of \eqref{direct sum decomposition preserved by} for the space $\V'$, the restrictions of  the elements $1$, $x'$, $x'{}^2$, $x'{}^3$,... to $\X'\oplus \V'{}^{ell}$ span $\End(\X'\oplus \V'{}^{ell})^{x'}$ over $Z(\Df)$.
\end{lem}
\begin{prf}
This is immediate from \eqref{gD3.2} or \eqref{csginGlD} respectively.
\end{prf}
 
We shall identify the cotangent bundle $T^*\a''$ with $\a''\times\a''{}^*$ and $T^*\A''$ with $\A''\times\a''{}^*$ and use the notation developed just before Lemma \ref{geometricII}. 
\begin{lem}\label{geometricI}
Let
\begin{eqnarray}\label{geometricI.1}
S_{\a''}&=&\{(x,\tau_{\a''{}^*}(w))\in \a''\times(\a''{}^*\setminus \{0\});\ x(w)=0\,,\ w\in \W\}\,,\\
S_{\A''}&=&\{(g,\tau_{\a''{}^*}(w))\in \A''\times(\a''{}^*\setminus \{0\});\ g(w)=-w\,,\ w\in \W\}\,.\nn
\end{eqnarray}
Fix an element $x'\in\reg{\h'}$ and $h'\in\reg{\H'}$. Then $S_{\a''}\cap N_{\iota_{\g, x'}}=\emptyset$
and $S_{\A''}\cap N_{\iota_{\G, h'}}=\emptyset$.
\end{lem}

\begin{prf}
Suppose $(x+x', \tau_{\a''{}^*}(w))\in N_{\iota_{\g, x'}}$. Then the restriction 
$\tau_{\a''{}*}(w)|_{\g}=0$. By the definition \eqref{moment_map_II}, 
$\tau_{\a''{}*}(w))|_{\g}=\tau_{\g^*}(w)$. In terms of the decomposition \eqref{A'isotypicdecompositionWI} 
\[
w=(w_0, w_1, w_2, \ldots)\,,\ \ w_0\in\W^{ell}\,,\ w_j=(x_j,y_j)\,,\ \ x_j\in\Hom(\V'_j,\V)\,,\ \ y_j\in\Hom(\V,\V'_j)\,,
\]
where $j\geq 1$.
Hence, by \eqref{identification_of_moment_maps_I}, 
\[
0=\tau_{\g^*}(w)(z)=\tr(z(w_0w_0^*+2x_1y_1^*+2x_2y_2^*+\ldots)) \qquad (z\in \g)\,.
\]
Hence,
\begin{equation} \label{eqx1yI1}
w_0w_0^*+2x_1y_1^*+2x_2y_2^*+\ldots =0\,.
\end{equation}
We also have the condition
\[
(x+x')w =0 \,,
\]
which means that for all $j$
\begin{equation}\label{(7.7)I}
xw_0=w_0x'\,,\ \ \ xx_j=x_jx'\,,\ \ \ y_jx=x'y_j\,,\ \ \ j\geq 1\,.
\end{equation}
Multiplying $x'{}^k$ from the left to \eqref{eqx1y1} and using \eqref{(7.7)} gives
\begin{equation}\label{(7.7)0I}
w_0x'{}^kw_0^*+2x_1x'{}^ky_1+2x_2x'{}^ky_2+ \ldots =0 \qquad (k=0, 1, 2, \ldots \,.
\end{equation}
Hence, Lemma \ref{regular elements generate I} implies that 
\[
w_0z'w_0^*+2x_1z'y_1^*+2x_2z'y_2^*+ \ldots =0 \qquad (z'\in \End(\X'\oplus \V'{}^{ell})^{x'})\,.
\]
Therefore
\begin{equation}\label{(7.7)1I}
w_0w_0^*=0\,,\ \ \ x_jy_j^*=0\,,\ \ \ j\geq 1\,,
\end{equation}
which, by \eqref{A''a''I.1} and \eqref{identification_of_moment_maps_I},  means that 
$\tau_{\a''^*}(w)=0$. Hence, $S_{\a''}\cap N_{\iota_{\g, x'}}=\emptyset$. The modification of the above argument to show that $S_{\A''}\cap N_{\iota_{\G, x'}}=\emptyset$ is similar to that used in the proof of Lemma \ref{geometricII}.
\end{prf}

\section{\bf The distributions $\chc_{x'} \in \Ss^*(\g)$ for $x'\in\reg{\h'}$.}\label{The distributionchcx'}

\subsection{\bf The minimal nilpotent $\Sp(\W)$-orbital integral in $\sp(\W)$.}\label{The minimal nilpotent}
\
Let $\W$ be a finite dimensional vector space over $\F$ with a non-degenerate symplectic form $\langle \cdot,\cdot\rangle$.
Let $\Sp(\W) \subseteq \End(\W)$ denote the corresponding symplectic group, with the Lie algebra $\sp(\W) \subseteq \End(\W)$. Fix a lattice $\mathcal L\subseteq \W$. We shall assume that  $\mathcal L$ is self-dual with respect to the symplectic form in the sense that
\begin{equation}\label{selfduallattice}
\left(\langle w,w'\rangle\in\fo_\F\ \ \text{for all}\ \ w'\in \mathcal L\right)\  \Longleftrightarrow\ w\in \mathcal L\,.
\end{equation}
We normalize the Haar measure $d\mu_\W(w)=dw$ on $\W$ so that the volume of $\mathcal L$ is $1$. (Note that ${\mathcal{L}}$ is open compact.) This determines a normalization of the Haar measure $\mu_\Uv$ on any subspace $\Uv\subseteq\W$, so that $\mu_\Uv(\Uv\cap\mathcal L)=1$. 

Recalling \cite[page 189]{Weil_Basic},  we let $\End(\W,\mathcal L)$ be the set of endomorphisms in $\End(\W)$ preserving $\mathcal L$. The set $\End(\W,\mathcal L)$ is a lattice in $\End(\W)$. We shall use this lattice to normalize the Haar measures on the subspaces of $\End(\W)$ as above, in particular on $\sp(\W)$.

Fix a unitary character $\chi$ of the additive group $\F$ whose kernel $\ker \chi = \fo_\F$. We recall the notion of the wave front set (see Appendix A). 
We define
\begin{equation}\label{TFofminimalnilpotentorbit1}
\chc(\psi)=\int_\W\int_{\sp(\W)}\chi(\langle x w, w\rangle) \psi(x)\,dx\,dw 
\end{equation}
for $\psi\in\Ss(\sp(\W)$.

\begin{thm}\label{Tminnilorb}
\begin{enumerate}[(i)]
\item Each consecutive integral in \eqref{TFofminimalnilpotentorbit1} converges absolutely and it defines a distribution $\chc \in \Ss^*(\sp(\W))$.

\item We have
\begin{equation}\label{TFofminimalnilpotentorbit2}
\chc(\psi)=\int_{\sp(\W)}\gamma_{{\mathrm{Weil}}}(x)|\det(x)|_\F^{-\frac{1}{2}} \psi(x)\,dx \qquad (\psi\in\Ss(\sp(\W))\,,
\end{equation}
where $\gamma_{{\mathrm{Weil}}}(x)$ is the Weil factor of the quadratic form $\langle x w, w\rangle$, \cite[Corollary 2, page 162]{WeilWeil}, $|\cdot|_\F$ is the module on $\F$, \cite[page 3]{Weil_Basic} and the integral is absolutely convergent. In particular $\gamma_{{\mathrm{Weil}}}(x)^8=1$. 

\item The subset
\begin{equation}\label{TFofminimalnilpotentorbit3}
\tau_{\sp^*(\W)}(\W\setminus\{0\})\subseteq \sp^*(\W)\,,
\end{equation}
is a minimal nilpotent coadjoint orbit.
The distribution $\chc$ defined in \eqref{TFofminimalnilpotentorbit1} is a Fourier transform of a positive invariant measure on this nilpotent orbit. Furthermore,
\begin{equation}\label{TFofminimalnilpotentorbit4}
\WF(\chc)=\{(x,\tau_{\sp^*(\W)}(w));\ \tau_{\sp^*(\W)}(w)\ne 0\,,\ x(w)=0\}\,.
\end{equation}
\end{enumerate}
\end{thm}

\begin{prf}
Let $N_{\mathcal L}(w)=\inf_{a\in\F^\times,\ aw\in\mathcal L}|a|_\F^{-1}$ be the norm associated to the lattice $\mathcal L$, \cite[page 28]{Weil_Basic}, and let $w_1$, $w_2$, ..., $w_{2n}$ be an $N_{\mathcal L}$-orthonormal basis of $\W$, in the sense that
$\W=\F w_1\oplus\F w_2\oplus\ldots\oplus \F w_{2n}$, \footnote{We may assume that for $i \leq j$, we have $\langle w_i, w_j \rangle = \delta_{i+n,j}$.} 
and $1=N_{\mathcal L}(w_1)=N_{\mathcal L}(w_2)=\dots N_{\mathcal L}(w_{2n})$. By taking the coordinates with respect to this basis we get an $\F$-linear isomorphism
\[
\W\ni w=u_1w_1+u_2w_2+\ldots u_{2n}w_{2n}\to u=(u_1, u_2, \ldots, u_{2n})\in M_{1,2n}(\F)
\]
which transports the measure $dw$ to $du=du_1du_2 \ldots du_{2n}$, where each $du_j$ is normalized so that the mass of $\fo_\F$ is $1$. Also, for each $x\in\sp(\W)$, the bilinear form $\langle x \cdot, \cdot\rangle$ is symmetric on $\W$. 
We define a symplectic form $\langle \cdot, \cdot \rangle$ on $M_{1,2n}(\F)$ such that $\langle e_i, e_j \rangle =\langle w_i, w_j \rangle$.
Hence, we have an $\F$-linear bijection
\[
\sp(\W)\ni x\to S(x)\in SM_{2n}(\F)\,,\ \ \ S(x)_{i,j}=\langle x e_i, e_j\rangle\,,
\]
which transports the normalized measure $dx$ on $\sp(\W)$ to a measure $dC$ on the space of the symmetric matrices $SM_{2n}(\F)$. Define the Fourier transform on this space by
\[
\hat f(D)=\int_{SM_{2n}(\F)} \chi(-trace(DC)) f(C)\,dC \qquad (f\in \Ss(SM_{2n}(\F)) \,,\ D\in SM_{2n}(\F))
\]
where $trace(\cdot)$ stands for the trace of a matrix. 
Notice that for $u,u'\in  M_{1,2n}(\F)$ the equality $u^tu=u'{}^tu'$ happens if and only if $u'=\pm u$. 
Hence we may define a measure $\nu$ on $SM_{2n}(\F)$ by
\[
\nu(f)=\int_{M_{1,2n}(\F)}f(-u^tu)\,du \qquad (f\in \Ss(SM_{2n}(\F))\,,
\]
where the integral is absolutely convergent.
Furthermore
\[
\hat\nu(f)=\nu(\hat f)=\int_{M_{1,2n}(\F)}\int_{SM_{2n}(\F)}\chi(uCu^t) f(C)\,dC\,du
\qquad (f\in \Ss(SM_{2n}(\F))
\]
because $uCu^t=trace(u^tuC)$. The above becomes \eqref{TFofminimalnilpotentorbit1} if we set
\begin{equation}\label{TFofminimalnilpotentorbit11}
\chc(\psi)=\hat\nu(\psi\circ S^{-1}) \qquad (\psi\in \Ss(\sp(\W))\,.
\end{equation}
For $x\in\sp(\W)$ and $w\in\W$, let $[X_{i,j}]$ be the matrix of $x$ with respect to the ordered basis $\{ e_1, e_2, ... \}$ of $\W$, 
\[
xe_j=\sum_i X_{i,j} e_i\,,
\]
and let $(u_1, u_2, \ldots, u_{2n})$ be the coordinates of $w$ as before, 
$
w=\sum_iu_iw_i
$.
Set
\[
Y_{i,j}=\sum_l \langle e_j,e_l\rangle u_iu_l \,.
\]
Then it is easy to check that the formula
\[
ye_j=\sum_i Y_{i,j} e_i
\]
defines an element $y\in\sp(\W)$. Furthermore,
\[
\sum_j Y_{i,j}Y_{j,k}=\sum_{j,l,l'}\langle e_j,e_l\rangle u_iu_l  \langle e_k,e_{l'}\rangle u_ju_{l'}
=\sum_{j}\langle e_j,w\rangle u_i  \langle e_k,w\rangle u_j=
\langle w,w\rangle u_i  \langle e_k,w\rangle=0\,.
\]
Therefore $y^2=0$ in $\End_\F(\W)$. Since $\W\setminus\{0\}$ is a single $\Sp(\W)$-orbit, 
$\tau_{\sp(\W)}(\W\setminus\{0\})$ is a single nilpotent coadjoint orbit. Its dimension is equal to the dimension of $\W$, hence is minimal possible, see \cite{CollMc}.
Thus $\chc$ is the Fourier transform of a minimal nilpotent orbit. 

The absolute convergence in (i) follows from \cite[Theorem 4.4]{DeBackerSally1999}.
The explicit formula for the integral \eqref{TFofminimalnilpotentorbit2} follows from \cite[Theorem 2, page 161]{WeilWeil}. More precisely, for an invertible $x\in\sp(\W)$, the function
\[
\W\ni w\to \chi(\langle x w, w\rangle)\in\C
\]
is called a character of the second degree in \cite{WeilWeil}. He computes the Fourier transform of it, which is the function 
\[
\W\ni w' \to \gamma_{{\mathrm{Weil}}}(x)|\det(x)|_\F^{-\frac{1}{2}}\chi(\langle -x^{-1} w', w'\rangle)\in\C\,.
\]
By evaluating the Fourier transform at $w'=0$ we see that
\[
\int_\W\chi(\langle x w, w\rangle) \,dw=\gamma_{{\mathrm{Weil}}}(x)|\det(x)|_\F^{-\frac{1}{2}}\,.
\]
By Theorem \ref{a1}, the wave front set of $\nu$ at $0$ is equal to $\supp \hat \nu$. As we have can see from \eqref{TFofminimalnilpotentorbit11} and \eqref{TFofminimalnilpotentorbit2}, this coincides with the whole space $SM_{2n}(\F)$.

From now on we consider $C\in \supp\,  \nu \setminus\{0\}$. 
The measure $\nu$ is invariant under the action
\[
SM_{2n}(\F)\ni C\to gCg^t\in SM_{2n}(\F) \qquad (g\in \SL_{2n}(\F)).
\]
Since $\supp \,\nu \setminus\{0\}$ is a single orbit under this action, Theorem \ref{Theorem 8.1.5Hormander1} 
applied to the submanifold $\supp \, \nu \setminus\{0\}\subseteq SM_{2n}(\F)$,
implies that $(C,D)\in WF(\nu)$ if and only if $D$ is orthogonal to the tangent space to $\supp\,\nu\setminus\{0\}$ at $C$. 
This means that for any $X\in \sl_{2n}(\F)$
\[
0=\tr(D(XC+CX^t))=2\tr(CDX)\,.
\]
Hence $CD$ is a constant multiple of the identity. 
However $C$ is of rank $1$ and it is not invertible. 
Thus $CD=0$. Equivalently $0=(CD)^t=DC$. Hence
\begin{equation}\label{wfftrankone1}
WF(\nu)=\{(C,D)\in \supp \nu\times SM_m(\F);\ DC=0\}
\end{equation}
Since $\nu$ is homogeneous,
\eqref{wfftrankone1} together with Theorem \ref{a1} imply \eqref{TFofminimalnilpotentorbit4}.  
\end{prf}

\subsection{\bf The minimal $\GL(\X)\times \GL(\X)$-orbital integral in $\End(\X)$.}\label{The minimal orbital}
\
Fix a complete polarization $\W=\X\oplus \Y$. The group $\F^\times$ acts on $\W$ by
\begin{equation}\label{Faction}
a(x+y)=xa+ya^{-1} \qquad (a\in \F^\times\,,\ x\in\X\,,\ y\in\Y)
\end{equation}
preserving the symplectic form $\langle\cdot,\cdot\rangle$. The centralizer of this action in $\Sp(\W)$ preserves $\X$ and $\Y$ and is isomorphic to $\GL(\X)$, by restriction.
The group $\F^\times$ acts on
\[
\W_{\F^\times}=(\X\setminus \{0\})\oplus (\Y\setminus \{0\})
\]
without fixed points. Define a positive measure $d(\F^\times w)$ on the orbit space $\F^\times\backslash \W_{\F^\times}$ by
\begin{equation}\label{measureonthequotient}
\int_{\W_{\F^\times}}\phi(w)\,dw=\int_{\F^\times\backslash \W_{\F^\times}}\int_{\F^\times}
\phi(a(w))\,\frac{da}{|a|}\,d(\F^\times w) \qquad (\phi\in\Ss(\W))\,.
\end{equation}
The group $\GL(\X)\times \GL(\X)$ acts on $\End(\X)$ by 
\[
(g,h) \cdot x = gxh^{-1}
\]
for $(g,h) \in \GL(\X)\times \GL(\X)$ and $x \in \End(\X)$.
This gives the corresponding left and right action on the function space
\[
L\otimes R(g,h)\psi(x)=\psi(g^{-1}xh)\qquad (g,h\in \GL(\V), x\in \End(\X), \psi\in\Ss(\End(\X))\,.
\]
Let $n=\dim_\F\X$. For each integer $0\leq k\leq n$, let $\Or_k$ be the set of endomorphisms in $\End(\X)$ of rank~$k$. The set $\Or_k$ is a single $\GL(\X)\times \GL(\X)$-orbit and up to a constant multiple there is a unique positive measure $\mu_{\Or_k}$ supported on $\Or_k$ such that
\begin{equation}\label{transformationproperty}
\mu_{\Or_k}(L\otimes R(g,h)\psi)=|\det(g)|^{k}|\det(h)|^{-k}\mu_{\Or_k}(\psi) \qquad (\psi\in C_c(\Or_k))\,.
\end{equation}

\begin{thm} \label{TFofminimalorbit}
(i) For $k=1$ and $k=n-1$, the measure $\mu_{\Or_k}$ extends by zero to define a distribution $\mu_{\Or_k}\in\Ss^*(\End(\X))$. This extension satisfies the transformation property \eqref{transformationproperty} for $\psi\in\Ss(\End(\X))$.

(ii) For $\psi\in\Ss(\End(\X))$, we have the following formula 
\begin{equation}\label{TFofminimalorbit1}
\mu_{\Or_{n-1}}(\psi) = \hat{\mu}_{\Or_1}(\psi) = \int_{\F^\times\backslash \W_{\F^\times}}\int_{{\mathrm{End}}(\X)}\chi(\langle x w, w\rangle) \psi(x)\,dx\,d(\F^\times w) 
\end{equation}
where each consecutive integral is absolutely convergent. The right hand side defines a distribution $\Ss^*(\End(\X))$.
\end{thm}

By Part (ii) of the above theorem, we define
\begin{equation} \label{TFofminimalorbit2}
\chc(\psi) = \int_{\F^\times\backslash \W_{\F^\times}}\int_{{\mathrm{End}}(\X)}\chi(\langle x w, w\rangle) \psi(x)\,dx\,d(\F^\times w) \qquad (\psi\in\Ss(\End(\X))).
\end{equation}
It is a distribution in $\Ss^*(\End(\X))$.

\begin{thm} \label{TFofminimalorbitWF}
We have
\[
\WF(\chc)=\{(x,\tau_{{\mathrm{End}}(\X)^*}(w));\ \tau_{{\mathrm{End}}(\X)^*}(w)\ne 0\,,\ x(w)=0\}\,.
\]
\end{thm}

\begin{rem}\label{TFofminimalorbitremark} 
The set where the derivative (i.e. the gradients) of $\det \colon \End(\V)\to\F$ is nonzero is equal to the subset where all of the minors of size $n-1$ are non-zero. This is the set of elements of rank at most $n-2$, i.e. $\Or_{n-2}\cup\Or_{n-3}\cup \ldots\cup\Or_0$. 
Hence the set where the derivative is non-zero is equal to $\End(\X)_{rk\geq n-1}=\Or_n\cup\Or_{n-1}$.
On this subset there is a well defined pullback $\det^*\delta_0=\delta_0\circ\det$. Explicitly, if $\phi_j\in C_c(\F)$, $\phi_j\geq 0$, $\int\phi_j(x)\,dx=1$ and the support of $\phi_j$ shrinks to zero if $j\to\infty$, then
\[
\det^*\delta_0(\psi)=\underset{{j\to\infty}}\lim\int_{{\mathrm{End}}(\X)} \phi_j(\det(x))\,\psi(x)\,dx
\qquad(\psi\in  C_c(\End(\X)_{rk\geq n-1}))\,.
\]
This follows from Corollary \ref{Hormender 8.2.3.cor} and Theorem \ref{Theorem 8.2.4}.
Since this distribution satisfies \eqref{transformationproperty}  for $k = n-1$,
it is convenient to think of $\chc$ as of a positive multiple of
$\delta_0\circ\det=\det_0^*\delta_0$: 
\begin{equation}\label{TFofminimalorbit5}
\chc(x)=c_0\delta_0(\det(x))\qquad (x\in \End(\X))\,.
\end{equation}
\end{rem}

\begin{prf}[ Proof of Theorem \ref{TFofminimalorbit}]

For $x=(x_1, \ldots , x_n)\in \F^n$ let 
\[
| x |=\max\{|x_1|, \ldots , |x_n|\} \,.
\]
The unit sphere 
$
S^{n-1}=\{x\in \F^n;\ |x|=1\}
$
may be thought of as a unit cube with decomposition into a disjoint union of facets 
\[
S^{n-1}_{i_1i_2\ldots i_k}=\{x\in \F^n;\ |x_{i_1}|=|x_{i_2}|=\ldots =|x_{i_k}|=1\,,\ |x_j|<1\,,\ j\notin\{i_1, i_2, \ldots , i_k\}\}\,,
\]
which are open subsets of $\F^n$:
\[
S^{n-1}=\bigcup_{k=1}^n\bigcup_{1\leq i_1<i_2<\ldots <i_k\leq n} S^{n-1}_{i_1i_2\ldots i_k}\,.
\]
Set
\[
\fo_\F^\times\backslash S^{n-1}_{i_1i_2\ldots i_k}=\{x\in \F^n;\ x_{i_1}=1\,,\ |x_{i_2}|=\ldots =|x_{i_k}|=1\,,\ |x_j|<1\,,\ j\notin\{i_1, i_2, \ldots , i_k\}\}
\]
and 
\[
\fo_\F^\times\backslash S^{n-1}=\bigcup_{k=1}^n\bigcup_{1\leq i_1<i_2<\ldots <i_k\leq n}
\fo_\F^\times\backslash S^{n-1}_{i_1i_2\ldots i_k}\,.
\]
Then the map
\[
\fo_\F^\times \times(\fo_\F^\times\backslash S^{n-1})\ni (a,x)\to ax\in S^{n-1}
\]
is bijective and
\[
\int_{S^{n-1}} \phi(x)\,dx=\int_{\fo_\F^\times\backslash S^{n-1}}\int_{\fo_\F^\times} \phi(ax)\,da\,dx
\qquad (\phi\in C(S^{n-1}))\,.
\]
Hence, we have the following formula for the ``integration in spherical coordinates",
\[
\int_{\F^n}\phi(x)\,dx
=\int_{\fo_\F^\times\backslash S^{n-1}}\int_{\F^\times} \phi(ax) |a|^n\,\frac{da}{|a|}\,dx
\qquad (\phi\in \Ss(\F^n))\,.
\]
Therefore,
\begin{eqnarray}\label{gl1gln2}
\lefteqn{\int_{\F^n}\int_{\F^n} \phi(x,y)\,dx\,dy 
=\int_{\F^n}\int_{\fo_\F^\times\backslash S^{n-1}}\int_{\F^\times} \phi(ax, y)\,|a|^n\,\frac{da}{|a|}\,\,dx\,dy} \nn\\
&=&\int_{\F^n}\int_{\fo_\F^\times\backslash S^{n-1}}\int_{\F^\times} \phi(ax, a^{-1}y)\,\frac{da}{|a|}\,\,dx\,dy\nn \quad \text{  (Substitute $y$ by $a^{-1}y$)} \\ 
&=&\int_{\fo_\F^\times\backslash S^{n-1}}\int_{\F^n}\left(\int_{\F^\times} \phi(ax, a^{-1}y)\,\frac{da}{|a|}\right)\,dy\,dx\,.
\end{eqnarray}
Define a measure $\mu_1$ on $\gl_n(\F)$ by
\begin{equation}\label{mesuremu1}
\mu_1(\psi)=\int_{\fo_\F^\times\backslash S^{n-1}}\int_{\F^n}\psi(x^ty)\,dy\,dx \qquad (\psi\in \Ss(\gl_n(\F)))\,.
\end{equation}
We shall compute the Fourier transform $\hat\mu_1$ of the measure $\mu_1$.  We have assumed in Section~\ref{The minimal nilpotent} that the kernel of the character $\chi$ is equal to $\fo_\F$ and that the volume of $\fo_\F$ is $1$. The Fourier transform of the Haar measure on $\F$ is the Dirac delta function at the origin:
\[
\phi(0)=\int_{\F}\int_{\F}\chi(-xy) \phi(x)\,dx\,dy \qquad (\phi\in \Ss(\F))\,.
\]
Let $\K=\GL_n(\fo_\F)$. This is a maximal compact subgroup of $\GL_n(\F)$. (In fact the unique maximal compact subgroup up to conjugation, see \cite[Appendix 1, Theorem 2, page 122]{SerreLieGroups}.) Let
\[
e_1=(1,0,0,\ldots , 0)\,.
\]
We normalize  the Haar measure $dk$ on the group $\K$ so that
\begin{equation}\label{normalizationofhaarmesureonK}
\mu_1(\psi)=\int_{\K}\int_{\F^n}\psi(ke_1^ty)\,dy\,dk \qquad (\psi\in \Ss(\g\l_n(\F)))\,.
\end{equation}
The Fourier transform of $\mu_1$ acting on a test function $\psi\in\Ss(\g\l_n(\F))$ may be computed as follows,
\begin{eqnarray}\label{gl1gln5}
\hat\mu_1(\psi)&=&\mu_1(\hat\psi) \nn \\
&=&\int_{\K}\int_{\F^n}\int_{\gl_n(\F)}\chi(-\tr(zke_1^ty))\psi(z)\,dz\,dy\,dk \\
&=&\int_{\K}\int_{\F^n}\int_{\gl_n(\F)}\chi(-\tr(ze_1^ty))\psi(zk)\,dz\,dy\,dk\nn\\
&=&\int_{\K}\int_{\F^n}\int_{\gl_n(\F)}\chi(-\sum_{j=1}^nz_{j,1}y_j)\psi(zk)\,dz\,dy\,dk\nn\\
&=&\int_{\K}\int_{\F^{n(n-1)}}\psi(
\left(
\begin{array}{llllll}
0&z_{1,2}&z_{1,3}& \ldots &z_{1,n}\\
0&z_{2,2}&z_{2,3}& \ldots &z_{2,n}\\
..&..&..& \ldots &..\\
0&z_{n,2}&z_{n,3}& \ldots &z_{n,n}\\
\end{array}
\right)k)\,dz\,dk\, .\nn
\end{eqnarray}
The last formula defines a positive measure supported on matrices of rank $n-1$ and shows that it extends by zero to a distribution on $\gl_n(\F)$. It is clear that this measure transforms according to \eqref{transformationproperty} with $k=n-1$, under the left translations.
We see from \eqref{mesuremu1} that measure $\mu_1$ transforms according to \eqref{transformationproperty} with $k=1$, under the right  translations. 
Therefore its  Fourier transform according to \eqref{transformationproperty} with $k=n-1$, under the right translations. This shows that $\hat\mu_1$ is a positive measure supported on $\Or_{n-1}$ that transforms the same way as the measure $\mu_{n-1}$ under the action of $\GL(\X)\times \GL(\X)$. Therefore 
with the correct normalization of the measure $\mu_{n-1}$ we have
\begin{equation}\label{gl1gln50-1}
\hat\mu_1=\mu_{n-1}.
\end{equation}
Since $|-2|_\F=1$, the change of variables, $y \mapsto -2y$ in \eqref{gl1gln5} shows that
\[
\hat\mu_1(\psi)=\int_{\K}\int_{\F^n}\int_{\gl_n(\F)}\chi(2\tr(zke_1^ty))\psi(z)\,dz\,dy\,dk\,.
\]
By \eqref{symplectic_form_II}
\begin{align*}
\langle z(ke_1^t,y),(ke_1^t,y)\rangle  & = \langle (zke_1^t,-yz),(ke_1^t,y)\rangle = \tr(zke_1^ty) + \tr(ke_1^t yz) \\
& = 2\tr(zke_1^ty).
\end{align*}
Hence,
\[
\hat\mu_1(\psi)=\int_{\K}\int_{\F^n}\int_{\gl_n(\F)}\chi(\langle z(ke_1^t,y),(ke_1^t,y)\rangle)\psi(z)\,dz\,dy\,dk\,.
\]
Equations \eqref{gl1gln2} and \eqref{normalizationofhaarmesureonK} imply that the integration over $\K\times \F^n$ is the same as the integration of an $\F^\times$-invariant function over the quotient $\F^\times\backslash \W_{\F^\times}$ 
Therefore 
\[
\hat\mu_1(\psi)=\int_{\F^\times\backslash \W_{\F^\times}}\int_{\gl_n(\F)}\chi(\langle z(w),w\rangle)\psi(z)\,dz\,d(\F^\times w)\,.
\]
By \eqref{TFofminimalorbit1}, the right hand side of the above equation is equal to $\chc(\psi)$. Together with \eqref{gl1gln50-1}, we get
\[
\chc = \hat\mu_1 = \mu_{n-1}.
\]
This proves \eqref{TFofminimalorbit1}. Hence the theorem follows.
\end{prf}

\begin{prf}[Proof of Theorem \ref{TFofminimalorbitWF}]
We compute the wave front sets of $\mu_1$ and $\mu_{n-1}$.
 Since $\mu_1$ is a homogeneous
distribution, Theorem \ref{a1} implies that the fiber of $\WF(\mu_1)$ over
$0\in\g\l_n(\F)$ coincides with $\supp \hat \mu_1=\g\l_n(\F)_{rk\leq n-1}$, the set of elements
of rank less or equal to $n-1$. The support of $\mu_1$ is the set of  matrices of rank at most $1$.
This is a single $\GL_n(\F)\times\GL_n(\F)$-orbit together with the zero matrix. Hence, by Theorem \ref{Theorem 8.1.5Hormander1},  a point $(x,y)\in (\supp \mu_1\setminus \{0\})\times \g\l_n(\F)$
belongs to $\WF(\mu_1)$ if and only if $y$ is perpendicular to the orbit through $x$:
\begin{eqnarray*}
\tr(zxy)=0\qquad (\text{for all } z\in \gl_n(\F))\,,\\
\tr(xzy)=0\qquad (\text{for all } z\in \gl_n(\F))\,.
\end{eqnarray*}
Hence $xy=0$ and, since $\tr(xzy)=\tr(yxz)$, $yx=0$ also. Thus
\begin{equation}\label{gl1gln6}
\WF(\mu_1)=\{(x,y)\in \g\l_n(\F)_{rk\leq 1}\times (\g\l_n(\F)_{rk\leq n-1}\setminus 0);\ xy=yx=0\}\,.
\end{equation}
Therefore Theorem \ref{a1} implies that
\begin{eqnarray}\label{gl1gln7}
\WF(\mu_{n-1})&=&\{(x,y)\in \g\l_n(\F)_{rk\leq n-1} \times (\g\l_n(\F)_{rk\leq 1}\setminus 0);\ xy=yx=0\}
\end{eqnarray}
Theorem \ref{TFofminimalorbitWF} follows the above calculations because $\chc = \mu_{\Or_{n-1}}$.
\end{prf}

\subsection{\bf The definition of $\chc_{x'}$.}\label{The definition}
Let $\G \cdot \G'\subseteq \Sp(\W)$ be an irreducible dual pair with the Lie algebra $\g \oplus \g'\subseteq \sp(\W)$. 
Let $\H'\subseteq \G'$ be a Cartan subgroup with the split part $\A'\subseteq \H'$, so that the quotient $\H'/\A'$ is compact. See sections \ref{A general linear group} and \ref{An isometry group}. 

\begin{lem}\label{question1}
Let $\A''\subseteq \Sp(\W)$ be the centralizer of $\A'$. Then $(\A'',\A')$ is a reductive dual pair in $\Sp(\W)$. 
Furthermore, there is an open dense $\A'$-invariant subset $\W_{\A'}\subseteq \W$ such that $\A'\backslash\W_{\A'}$ is a equipped with the $\A''$-invariant measure $d(\A' w)$ defined by 
\begin{eqnarray}\label{measureonthequotient0}
\int_\W\phi(w)\,dw=\int_{\A'\backslash\W_{\A'}}\int_{\A'} \phi(aw)\,da\,d(\A' w) \qquad (\phi\in \Ss(\W))\,.
\end{eqnarray}
\end{lem}
\begin{prf}
See \eqref{A''a''},  \eqref{A''a''I.1} and \eqref{measureonthequotient}. 
\end{prf}
Define a unitary Gaussian
\begin{eqnarray}\label{chix}
\chi_x(w)=\chi(\frac{1}{4}\langle x(w),w\rangle)\qquad (x\in\sp(\W),\ w\in \W)\,.
\end{eqnarray}
\begin{lem}\label{question2}
Let $\a''$ be the Lie algebra of $\A''$. Then
\begin{eqnarray}\label{question2.1}
\int_{\A'\backslash\W_{\A'}}\left|\int_{\a''}\psi(x)\chi_x(w)\,dx\right| \,d(\A' w)<\infty \qquad (\psi\in\Ss(\a''))
\end{eqnarray}
so the formula
\begin{equation}\label{question2.2a}
\chc(\psi)=\int_{\A'''\backslash\W_{\A'''}}\int_{\a''}\psi(x)\chi_x(w)\,dx\,d(\A' w) \qquad (\psi\in\Ss(\a''))
\end{equation}
defines a distribution on $\a''$. The wave front set of this distribution is equal to
\begin{eqnarray}\label{question2.2}
\WF(\chc)= \left\{(x,\tau_{\a''}(w)) \in \a'' \times (\a''^*\setminus 0) :\ x(w)=0\,,\ w\in \W \right\} \,.
\end{eqnarray}
\end{lem}
\begin{prf}
This is immediate from \eqref{A''a''},  \eqref{A''a''I.1} and 
{Theorems \ref{Tminnilorb}}, \ref{TFofminimalorbit},   and \ref{TFofminimalorbitWF}.
\end{prf}
\begin{pro}\label{proposition1.8}
Fix an element $x'\in \reg{\h'}$. Then the intersection of the wave front set $\WF(\chc)$ with the conormal bundle to the embedding
\begin{equation}\label{proposition 1.8.1.}
\g\ni x\to x'+x\in \a''
\end{equation}
is empty.
\end{pro}
\begin{prf}
By \eqref{question2.2}, $\WF(\chc)$ is equal to the set $S_{\a''}$, defined in \eqref{geometricII.1} and \eqref{geometricI.1}. The proposition follows from Lemmas \ref{geometricII} and \ref{geometricI}.
\end{prf}

Standard micro-local analysis, i.e. Theorem \ref{Theorem 8.2.4}, together with Proposition \ref{proposition1.8}
justify the following definition.

\begin{defi}\label{Definition 1.9.}
For $x'\in \reg{\h'}$ define $\chc_{x'}$ to be the pullback of the distribution $\chc$ to  $\g$ via the embedding
\[
\g\ni x\to x'+x \in \a''\,.
\]
Formally
\begin{equation}\label{question2.2ag}
\chc_{x'}(x)=\chc(x'+x)=\int_{\A'\backslash \W_{\A'}} \chi_{x'+x}(w)\,d(\A' w) \qquad (x\in \g)\,.
\end{equation}
\end{defi}

\begin{lem}\label{Lemma 1.10.}
For any $x'\in \reg{\h'}$ 
\[
\WF(\chc_{x'})\subseteq\{(x,\tau_{\g^*}(w));\ (x'+x)(w)=0,\ x\in\g,\ w\in \W\}\,.
\]
\end{lem}

\begin{prf}
This follows from the Definition \ref{Definition 1.9.}, Lemma \ref{question2} and Theorem \ref{Theorem 8.2.4}. Note that the right hand side is the image of $\WF(\chc)$ the under projection map $T_{x'+x}^* \a'' \rightarrow T_x^* \g$.
\end{prf}

\section{\bf The distributions $\Chc_{h'}\in C_c^\infty{}^*(\wt{\G})$, $h'\in\reg{\wt{\H'}}$.}\label{The distributionChcx'}

\subsection{\bf The character of the Weil representation of $\Sp(\W)$.}\label{The character of the Weil representation}
\ 
Recall the lattice $\mathcal L\subseteq \W$ fixed in section \ref{The minimal nilpotent}.
For an element $g\in \Sp(\W)$ choose a basis $\{ u_1, u_2, \ldots, u_m \}$ of $(g-1)\W$, such that
\[
\mathcal L\cap (g-1)\W =\fo_\F u_1+\fo_\F u_2+\ldots +\fo_\F u_m\,.
\]
For the existence of such a basis see \cite[Theorem 1, page 29]{Weil_Basic}. Let $w_1$, $w_2$, ..., $w_m$ be elements of $\W$ such that 
\[
\langle u_j,w_k\rangle =\delta_{j,k}\,.
\]
We recall the Gaussian
\[
\gamma(a)=\int_\F\chi(\frac{1}{2}ax^2)\,dx \qquad (a\in \F^\times)\,.
\]
Define the following function on the symplectic group,
\begin{equation}\label{Thetasquare}
\Theta^2(g)=\gamma(1)^{2\dim \W}\gamma(1)^{-2}\gamma(\det(\langle (g-1)w_j,w_k\rangle)_{1\leq j,k\leq m})^2\qquad (g\in \Sp(\W))\,.
\end{equation}
(Here the square $\Theta^2$ is part of the name of this function. At this point we don't have any function $\Theta$ such that $\Theta^2=(\Theta)^2$.)
One may check that the above definition does not depend on the choices of the $u_j$ and $w_k$ we made, see \cite[Definition 5.15]{AubertPrzebinda_omega}.
Let
\[
\wt{\Sp}(\W)=\{(g,\xi)\in \Sp(\W)\times\C^\times;\ \xi^2=\Theta^2(g)\}
\]
Then there is a cocycle $C$ on the symplectic group such that the formula
\[
(g_1,\xi_1)(g_2,\xi_2)=(g_1g_2,\xi_1\xi_2C(g_1,g_2))
\]
defines a group structure on $\wt{\Sp}(\W)$. The cocycle satisfies the following equation
\[
C(g_1,g_2)^2=\frac{\Theta^2(g_1g_2)}{\Theta^2(g_1)\Theta^2(g_2)} \qquad (g_1, g_2\in \Sp(\W))\,,
\]
see \cite[Lemma 5.16]{AubertPrzebinda_omega} and is explicitly described in \cite[Proposition 5.12]{AubertPrzebinda_omega}. Moreover the map
\[
\wt{\Sp}(\W)\ni (g,\xi)\to g\in \Sp(\W)
\]
is a group homomorphism with each fiber consisting of two elements. Thus $\wt{\Sp}(\W)$ is a double cover of $\Sp(\W)$ and is called the metaplectic group. 
For any $g\in \Sp(\W)$ define
\[
c(g)u=(g+1)w\,,\ \ \ u=(g-1)w\,,\ \ \  w\in\W
\] 
and
\[
\chi_{c(g)}(u)=\chi(\frac{1}{4}\langle c(g)u,u\rangle)\qquad (u\in (g-1)\W)\,.
\]
(If $g-1$ is invertible, then $c(g)=(g+1)(g-1)^{-1}$ is the Cayley transform, a birational isomorphism 
$c:\Sp(\W)\to \sp(\W)$, equal to its own inverse.)  
For any subspace $\Uv\subseteq \W$ let $\mu_\Uv$ be the Haar measure on $\Uv$ normalized so that the volume of $\mathcal L\cap \Uv$ is $1$. Define the following functions
\begin{eqnarray}
& & \Theta: \wt{\Sp}(\W)\ni (g,\xi)\to \xi\in \C^\times\,, \nn \\
& & T \colon \wt{\Sp}(\W)\ni (g,\xi)\to \xi\chi_{c(g)}\mu_{(g-1)\W}\in\Ss^*(\W)\,.
\label{eqdefnT}
\end{eqnarray}
The complex valued measure $\chi_{c(g)}\mu_{(g-1)\W}$ is viewed as a distribution on $\W$.
The map~$T$ is an injection of the metaplectic group into the space of the tempered distributions on~$\W$ with the following properties,
\begin{eqnarray*}
T(\t g_1)\natural T(\t g_2)&=&T(\t g_1 \t g_2) \qquad (\t g_1, \t g_2\in \wt{\Sp}(\W))\\
T(\t g)^*&=&T(\t g^{-1}) \qquad (\t g\in \wt{\Sp}(\W)) \\
T(1)&=&\delta_0\,,
\end{eqnarray*}
where the twisted convolution 
\[
\phi\natural \psi(w')=\int_\W \phi(w)\psi(w'-w)\chi(\frac{1}{2}\langle w,w'\rangle)\,d\mu_\W(w)
\qquad (\phi, \psi\in \Ss(\W))
\]
extends by continuity to some tempered distributions so that $T(\t g_1)\natural T(\t g_2)$ makes sense, see \cite[Lemm 5.23]{AubertPrzebinda_omega}. In particular
\begin{equation}\label{CocycleAndCharacter}
C(g_1,g_2)=\frac{\Theta(\t g_1\t g_2)}{\Theta(\t g_1)\Theta(\t g_2)} \qquad (\t g_1, \t g_2\in \wt{\Sp}(\W))\,,
\end{equation}
see \cite[ Proof of Lemm 5.23]{AubertPrzebinda_omega}. In addition, if $g_1$, $g_2$ and $g_1g_2$ are in the domain of the Cayley transform, then
\begin{equation}\label{CocycleAschc}
C(g_1,g_2)=\int_\W\chi(\frac{1}{4}\langle (c(g_1)+c(g_2))(w),w\rangle \,dw\,,
\end{equation}
see \cite[Proposition 5.12]{AubertPrzebinda_omega} specialized to this particular case.

Furthermore, we set
\[
\phi^*(w)=\overline{\phi(-w)}\,,\ \ \ u^*(\phi)=u(\phi^*)\qquad (w\in\W\,,\ \phi\in\Ss(\W)\,,\ u\in \Ss^*(\W))\,.
\]
In particular, if $T(\t g)$ is a function, then $T(\t g)^*(w) = \overline{T(\t g)(-w)}$.
It turns out that $\Theta$ is a locally integrable function that coincides with the distribution character of the Weil representation, see \cite[Theorem 5.26]{AubertPrzebinda_omega}, a direct sum of two irreducible components. 

Suppose  $\G$ is an analytic Lie group with the Lie algebra $\g$ over $\F$. (Here we think in terms of locally analytic functions, \cite[Part II]{SerreLieGroups}.)
Each $Y\in\g$ defines a left invariant vector field on $\G$ via the right regular representation, 
\begin{equation}\label{001}
R_Y\Psi(x)=\frac{d}{dt}\left.\Psi(x\exp(tY))\right|_{t=0} \qquad (x\in\G,\ \Psi\in C_c^\infty(\G))\,.
\end{equation}
Given a local chart
\[
\kappa:\G^\kappa\to \g^\kappa\,,
\]
where $\g^\kappa$ is an open subset of $\g$.
We can lift the vector field \eqref{001} to $\g^\kappa$ by the formula
\[
\kappa_*(R_Y)\psi(X)=R_Y(\psi\circ\kappa)(\kappa^{-1}(X))
 \qquad (X\in\g^\kappa,\ \psi\in C_c^\infty(\g^\kappa))\,.
\]
Explicitly,
\[
\kappa_*(R_Y)\psi(X)=\frac{d}{dt}\left.\psi(X+tv_X(Y))\right|_{t=0}\,,
\]
where
\begin{equation}\label{vX}
v_X(Y)= \frac{d}{dt}\left.\kappa(\kappa^{-1}(X)\exp(tY))\right|_{t=0}\,.
\end{equation}
Thus we have the following bijection
\begin{equation}\label{111}
\G^\kappa\times \g\ni (x,Y)\leftrightarrow (\kappa(x),v_{\kappa(x)}(Y))\in \g^\kappa\times\g\,.
\end{equation}
By dualizing \eqref{111} we get the following bijection
\begin{equation}\label{011}
\G^\kappa\times \g^*\ni (x,v_{\kappa(x)}^*\eta)\leftrightarrow (\kappa(x),\eta)\in \g^\kappa\times\g^*\,,
\end{equation}
where
\[
v_{\kappa(x)}^*\eta=\eta\circ v_{\kappa(x)}\,.
\]
According to \cite[Definition 6.4.1]{Hormander} an atlas on the tangent and cotangent bundle may be seen as
\[
\bigcup_{x\in\G^\kappa}T_x(\G)=\{(\kappa(x),Y);\ x\in \G^\kappa,\ Y\in \g\}
\]
and therefore
\[
\bigcup_{x\in\G^\kappa}T_x^*(\G)=\{(\kappa(x),\eta);\ x\in \G^\kappa,\ \eta\in \g^*\}\,.
\]
Hence the bijections \eqref{011} define 
an identification $\G\times \g^*=T^*\G$ such that for a given $\F$-valued compactly supported smooth function $\Psi$ we regard $d\Psi$ as a $\g^*$-valued function on $\G$ by
\begin{equation}\label{002}
d\Psi(x)(X)=R_Y\Psi(x) \qquad (x\in \G,\ X\in \g)\,,
\end{equation}
as in \cite{HoweWave}. In particular, under this identification, 
\begin{equation}\label{022}
\WF(u)=\{(x, v_{\kappa(x)}^*\eta);\ (\kappa(x),\eta)\in \WF((\kappa^{-1})^*u)\} \qquad (u\in D'(\G^\kappa))\,,
\end{equation}
where $(\kappa^{-1})^*u$ is a distribution on $\g^\kappa$. We now specialize the above to $\G = \wt{\Sp}(\W)$ and $\g = \sp(\W)$.

\begin{pro}\label{WFTheta}
$\WF(\Theta)=\{(g,\tau_{\sp^*(\W)}(w));\ g\in \wt{\Sp}(\W),\ w\in \W\setminus\{0\},\ gw=w\}$.
\end{pro}
\begin{prf}
Define the translations
\[
\Lambda_{\t x_0}:\wt{\Sp}(\W)\ni \t x\to \t x_0\t x\in \wt{\Sp}(\W) \qquad (\t x_0\in \wt{\Sp}(\W))
\]
and 
\[
\lambda_{X_0}:\g\ni X\to X_0+X\in\g \qquad (X_0\in \g)\,.
\]
Let $\Sp^c(\W)=\{g\in\Sp(\W);\ \det(g-1)\ne 0\}$ be the domain of the Cayley transform. 
We fix an element $\t x_0\in \wt{\Sp^c}(\W)$ and consider the distribution $u\in \Ss^*(\wt{\Sp^c}(\W))$ defined by
\[
u(\t x)=\Theta(\t x_0\t x)\qquad (\t x\in\wt{\Sp^c}(\W))
\]
times the Haar measure. In terms of the translation defined above,  $u=\Lambda_{x_0}^*\Theta|_{\Sp^c(\W)}$. Set $X_0=\t c(x_0)$. 
Then \eqref{CocycleAndCharacter} and \eqref{CocycleAschc} imply that as a distribution on $\sp(\W)^c$,
\begin{equation}\label{(12.6)}
\t c^* u =\Lambda_{x_0}^*\Theta|_{\Sp^c(\W)} =\Theta(x_0)\cdot \t c^*\Theta\cdot \lambda_{X_0}^*\chc,
\end{equation}
where $\chc$ is defined in  \eqref{TFofminimalnilpotentorbit1}. 
Equation \eqref{(12.6)} shows that
\[
\WF(\t c^*u) = \WF(\lambda_{X_0}^*\chc).
\]
Hence, \eqref{TFofminimalnilpotentorbit4} implies that
\begin{equation}\label{(12.60)}
\WF(\t c^*u) = \{(X,\tau_{\sp^*(\W)}(w));\ X\in\g^c,\ w\in \W,\ (X_0+X)w=0\}.
\end{equation}
A straightforward computation shows that
\[
v_{X}(Y)=\frac{1}{2}(1+X)Y(1-X) \qquad (Y\in\g),
\]
where $v_X$ was defined in \eqref{vX} with $\kappa=\t c$. Hence
\[
v_{X}^*\tau_{\sp^*(\W)}(w)(Y)=\langle \frac{1}{2}(1+X)Y(1-X)w,w\rangle=\frac{1}{2}\langle Y(1-X)w,(1-X)w\rangle
\]
Therefore
\begin{eqnarray}\label{(12.61)}
\WF(u) & = & \{(x,v_{X}^*\tau_{\sp^*(\W)}(w));\ X=\t c(x),\ w\in \W,\ (X_0+X)w=0\}\\
&=&\{(x,\tau_{\sp^*(\W)}((1-X)w));\ X=\t c(x),\ w\in \W,\ (X_0+X)w=0\}.\nn
\end{eqnarray}
However,
\[
X_0+X=2(x_0-1)^{-1}(x_0x-1)(x-1)^{-1}\ \ \ \text{and}\ \ \ (x-1)^{-1}=\frac{1}{2}(X-1).
\]
Hence, the equation $(X_0+X)w=0$ is equivalent to $(x_0x-1)(X-1)w=0$. Thus
\[
\WF(u)=\{(x,\tau_{\sp^*(\W)}(w'));\ w'\in \W,\ x_0x w' =w' \}
\]
and consequently
\[
\WF(\Theta)=\{(x_0x,\tau_{\sp^*(\W)}(w'));\ w'\in \W,\ x_0x w'=w',\ x_0 \in\Sp^c(\W),\ x  \in\Sp^c(\W) \}\,.
\]
Since, by Lemma \ref{GcGc} below, $\Sp^c(\W)\cdot\Sp^c(\W)=\Sp(\W)$, we are done.
\end{prf}
\begin{lem}\label{GcGc}
$\Sp(\W)=\Sp^c(\W)\cdot\Sp^c(\W)$.
\end{lem}
\begin{prf}
The subset $\Sp^c(\W)\subseteq \Sp(\W)$ is Zariski open and so is $g^{-1}\Sp^c(\W)$ for an $g\in \Sp(\W)$. 
Hence, for each $g\in \Sp^c(\W)$ there is $h\in\Sp^c(\W)$ such that $gh\in \Sp^c(\W)$.

Let $g\in\Sp(\W)$ be such that $\W_1\subseteq \W$, the $1$-eigenspace for $g$, is non-degenerate, i.e. 
\[
\W=\W_1\oplus\W_1^\perp.
\]
Then $g|_{\W_1^\perp}\in \Sp^c(\W_1^\perp)$. Let $k\in\Sp(\W)$ be defined by
\[
k(w_1+w_2)=-w_1+hw_2 \qquad (w_1\in\W_1,\ w_2\in \W_1^\perp)\,,
\]
where $h\in \Sp(\W_1^\perp)$ is such that $g|_{\W_1^\perp}h\in \Sp^c(\W_1^\perp)$. Then $k^{-1}, kg\in 
\Sp(W)^c$ and $g=k^{-1} (kg)$. 

Let $g\in \Sp(W)$ be such that $\X\subseteq \W$, the $1$-eigenspace for $g$, is maximal isotropic. 
Then 
\[
\Im(g-1)=\Ker(g-1)^\perp=\X^\perp=X\,.
\]
Hence $g=1$ on $\X$ and on $\W/\X$. Therefore $-g=-1$ on $\X$ and on $\W/\X$. Thus $-g\in \Sp(W)^c$. Since $g=(-1)(-g)$ the claim follows.

In general, there is a decomposition 
\[
\W = (\X \oplus \Y) \oplus \W_1 \oplus \W_2
\]
where $\X$, $\Y$ are isotropic subspaces, the restriction of the symplectic form to $\W_1$ and $\W_2$ are non-degenerate, and the $1$-eigenspace for $g$ is 
\[
\X\oplus \W_1\,.
\]
Here $\X$ is the radical of the symplectic form restricted to the $1$-eigenspace of $g$. Now $g$ preserves $(\X \oplus \W_1)^\perp = \X \oplus \W_2$. Since $g = 1$ on $\X$, $g$ preserves $\W_2$.
Likewise $g = 1$ on $\W_1$ so $g$ preserves $\W_1 \oplus \W_2$ and $(\W_1 \oplus \W_2)^\perp = \X \oplus\Y$.
We could write $g = g_0 g_1 g_2$ where $g_0$, $g_1$, $g_2$ are the restrictions of $g$ to $\X \oplus \Y$, $\W_1$ and $\W_2$ respectively.
Note that $g_2 \in \Sp(\W_3)^c$. We have treated $g_0$ and $g_1$ above.
This proves the lemma.
\end{prf}
\begin{thm}\label{theChc}
For any $\Psi\in C_c^\infty(\wt{\Sp}(\W))$ the distribution
\begin{equation}\label{theChc1}
T(\Psi)=\int_{\wt{\Sp}(\W)}T(g)\Psi(g)\,dg\in \Ss^*(\W)
\end{equation}
is a function  on $\W$ that belongs to $\Ss(\W)$ (times the measure $dw$) and the formula
\begin{equation}\label{theChc2}
\Chc(\Psi)=\int_\W T(\Psi)(w)\,dw
\end{equation}
defines a distribution of $\wt{\Sp}(\W)$. Thus in the space of distributions on $\wt{\Sp}(\W)$,
\begin{equation}\label{theChc3}
\Chc=\int_\W T(w)\,dw\,.
\end{equation}
Explicitly,
\begin{equation}\label{theChc4}
\Chc=\Theta(\wt{-1})^{-1}\Lambda^*_{\wt{-1}}\Theta\,.
\end{equation}
Furthermore,
\begin{equation}\label{theChc5}
\WF(\Chc) = \{(g,\tau_{\sp^*(\W)}(w));\ g\in \wt{\Sp}(\W),\ w\in \W\setminus\{0\},\ gw=-w\}\,.
\end{equation}
\end{thm}
\begin{prf}
Lemma \ref{GcGc} implies that there are $g_1, g_2, \ldots, g_m$ in $\wt{\Sp^c}(\W)$ such that
\[
\wt{\Sp}(\W)=\bigcup_{i=1}^m g_i\wt{\Sp^c}(\W)\,.
\]
Therefore, there are smooth functions $\Psi_i \in C^\infty(\wt{\Sp}(\W))$ with support $\Psi_i$ contained in $\wt{\Sp^c}(\W)$ such that
\[
\sum_{i=1}^m\Psi_i(g_i^{-1}g)=1 \qquad (g\in \wt{\Sp}(\W))\,.
\]
Hence, for $\Psi\in C_c^\infty(\wt{\Sp}(\W))$,
\begin{align}
T(\Psi) & =\sum_{i=1}^m\int_{\wt{\Sp}(\W)}\Psi_i(g_i^{-1}g)\Psi(g)T(g)\,dg=\sum_{i=1}^m\int_{\wt{\Sp}(\W)}\Psi_i(g)\Psi(g_ig)T(g_ig)\,dg \nn \\
& =\sum_{i=1}^mT(g_i)\natural\int_{\wt{\Sp^c}(\W)}\Psi_i(g)\Psi(g_ig)T(g)\,dg\,. \label{eqTPsi}
\end{align}
For $g\in \wt{\Sp^c}(\W)$, $T(g)(w) = \Theta(g) \chi(\frac{1}{4} \langle \t c(g)w, w \rangle)$ is a function  on $\W$ by \eqref{eqdefnT}.
For an appropriate $\psi_i\in C_c^\infty(\sp(\W))$,
\[
\int_{\wt{\Sp^c}(\W)}\Psi_i(g)\Psi(g_ig)T(g)(w)\,dg=\int_{\sp(\W)}\psi_i(x)\chi(\langle x(w),w\rangle)\,dx\,.
\]
This is the pullback of the Fourier transform of $\psi_i$ from $\sp^*(\W)$ to $\W$ via the moment map $\tau_{\sp^*(\W)}$ whose fibers have at most two elements. Hence this pullback is a function in $\Ss(\W)$.
The twisted convolution by $T(g_i)$ is a continuous linear map on $\Ss(\W)$ so \eqref{eqTPsi} is a function in $\Ss(\W)$. Hence \eqref{theChc1},  \eqref{theChc2} and \eqref{theChc3} follow. 

Since the Cayley transform maps $-1$ to $0$, $c(-1)=0$, we see that $T(\wt{-1})$ is a constant function:
\[
T(\wt{-1})(w)=\Theta(\wt{-1})\chi(\frac{1}{4}\langle c(-1)(w),w\rangle)=\Theta(\wt{-1})\,.
\]
If we view distributions as generalized functions, then we get
\begin{multline*}
\Lambda_{\wt{-1}}\Theta(g)=\Theta(\wt{-1}g)=T(\wt{-1}g)(0)
=T(\wt{-1})\natural T(g)(0)\\
=\int_{\W}T(\wt{-1})(0-w)T(g)(w)\,dw
=\int_{\W}\Theta(\wt{-1})T(g)(w)\,dw=\Theta(\wt{-1})\Chc(g)\,.
\end{multline*}
This verifies \eqref{theChc4}. Finally, \eqref{theChc5} follows from \eqref{theChc4} and Proposition \ref{WFTheta}.
\end{prf}
\subsection{\bf Induction from a mirabolic subgroup of $\GL(\X)$.}\label{Induction from a mirabolic subgroup}
\
Let $\Pg$ be a parabolic subgroup of a reductive group $\G$, with Langlands decomposition $\Pg=\L\N$. Let $\n$ be the Lie algebra of the unipotent radical $\N$ and let $\K\subseteq \G$ be an open compact subgroup of $\G$ such that the Iwasawa decomposition $\G = \K \L \N$ holds, see \cite[Theorem 5, page 16]{Harish-Chandra-van-Dijk}. 
For any function $\Psi\in C_c^\infty(\G)$ define
\begin{eqnarray} \label{eqq11.1}
&&\Psi^\K(g)=\int_\K\Psi(kgk^{-1})\,dk,\ \ \Psi^\K_\N(g)=\int_\N\Psi^\K(gn)\,dn\, \nonumber \\
&&\Psi^\L(l)=|\det(\Ad(l)_\n)|^{1/2}\Psi^\K_\N(l)
\qquad (g\in \G,\ l\in \L)\,. 
\end{eqnarray}
Clearly
$$ 
C_c^\infty(\G)\ni \Psi\to \Psi^\L|_{\L}\in C_c^\infty(\L) \label{eqq11.2}
$$
is a well defined continuous linear map. For a distribution
$u\in C_c^\infty(\L)^*$ we define the induced  distribution
$\Ind_\L^\G(u)\in C_c^\infty(\G)^*$ by
\begin{equation}\label{induceddistribution}
\Ind_\L^\G(u)(\Psi)=u(\Psi^\L)
\qquad (\Psi\in C_c^\infty(\G))\,.
\end{equation}
This construction may be found in \cite[page 18 and 30]{Harish-Chandra-van-Dijk}.
In the case of a real reductive group $\G$, it is explained in detail in \cite[part II, pages 121-123]{Varada}. See also \cite[Proposition A.5]{BerPrzeNo}.

Recall the complete polarization $\W=\X\oplus \Y$ in section \ref{The minimal orbital}. We shall identify the subgroup of the symplectic group preserving both $\X$ and $\Y$ with $\GL(\X)$. Denote by $\K\subseteq \GL(\X)$ the subgroup that preserves $\mathcal L\cap\X$. This is a maximal compact subgroup.  

Let $\X'\subseteq \X$ be a one dimensional subspace and let $\X''\subseteq \X$ be a complementary subspace so that $\X=\X'\oplus \X''$. Denote by $\Pg\subseteq \GL(\X)$ the subgroup of elements preserving $\X'$. This is known as a mirabolic subgroup, \cite{Bernstein1984}. The unipotent radical of $\Pg$ coincides with $\N=1+\Hom(\X'',\X')\subseteq\End(\X)$ and the Lie algebra $\n=\Hom(\X'',\X')$. The Levi factor $\L\subseteq \Pg$ consists of elements that preserve both $\X'$ and $\X''$. It is isomorphic to $\GL(\X')\times\GL(\X'')$. 
By choosing a non-zero vector in $\X'$ we identify $\X'$ with $\F$ and $\GL(\X')$ wit $\F^\times$.  Let $\wt{\GL}(\X)$ denote the inverse image of $\GL(\X)$ in $\wt{\Sp}(\W)$.
Define
\[
sum:C_c^\infty(\wt{\GL}(\X))\to C_c^\infty(\GL(\X))
\]
to be the sum over the fibers of the covering map
\[
sum(f)(g)=\sum_{\t g} f(\t g) \qquad (g\in \GL(\X))\,.
\] 
\begin{thm}\label{theChcGL}
For any $\Psi\in C_c^\infty(\wt{\GL^c}(\X))$, the distribution
\begin{equation}\label{theChc0GL}
T(\Psi)=\int_{\wt{\GL}(\X)} \Psi(\t g)T(\t g)\,d\t g\in \Ss^*(\W)
\end{equation}
is a function on $\W$, invariant under the \eqref{Faction} action of $\F^\times$, such that 
\[
\int_{{\F^\times}\backslash \W_{\F^\times}}\left|T(\Psi)(w)\right|\,d(\F^\times w)<\infty \,.
\]
The formula 
\begin{equation}\label{theChc2GL}
\Chc(\Psi)=\int_{{\F^\times}\backslash \W_{\F^\times}} T(\Psi)(w)\,d(\F^\times w)
\qquad (\Psi\in C_c^\infty(\wt{\GL^c}(\X)))
\end{equation}
defines a distribution of $\wt{\GL^c}(\X)$ which coincides with a complex valued measure
\begin{equation}\label{theChc3GL}
\Chc(\Psi)=\int_{\GL(X)} sum(\Psi\Theta)(g) |\det(g-1)| \delta(\det(g+1))\,dg\,.
\end{equation}
This measure extends by zero to $\wt{\GL}(\X)$ and defines a distribution, which shall be denoted by the same symbol, \eqref{theChc2GL}.
The function
\begin{equation}\label{theChc5GL}
\epsilon:\wt{\GL}(\X)\ni \t g\to \frac{\Theta(\t g)}{|\Theta(\t g)|}\in \C^\times
\end{equation}
is a group homomorphism and 
\begin{equation}\label{theChc6GL}
\Chc= \Ind_{\GL(\X')\times\GL(\X'')}^{\GL(\X)}\left( (\delta_1\otimes \mu) \circ \Lambda_{-1}\circ sum\circ mult_\epsilon \right) \, ,
\end{equation}
where $\delta_1$ is the Dirac delta at $1\in\GL(\X')$ and $\mu$ is the Haar measure on  $\GL(\X'')$ and $mult_\epsilon$ is the multiplication by $\epsilon$.
Furthermore, 
\begin{equation}\label{theChc7GL}
\WF(\Chc) = \{(\t g,\tau_{{\mathrm{End}}(\X)^*}(w));\ \t g\in \wt{\GL}(\X)\,,\ \tau_{{\mathrm{End}}(\X)^*}(w)\ne 0\,,\ g(w)=-w\}\,,
\end{equation}
\end{thm}

\begin{prf}
If $g\in \GL^c(\X)$ and $x=c(g)\in\End(\X)$, then $\det(x-1)=\det(2(g-1)^{-1})\ne 0$ and 
$\det(x+1)=\det(2g(g-1)^{-1})\ne 0$. Let $\Psi\in C_c^\infty(\wt{\GL^c}(\X))$. 
Then
\begin{equation}\label{theChc8GL}
T(\Psi)(w)=\int_{\wt{\GL}(\X)} \Psi(\t g)\Theta(\t g)\chi_{c(g)}(w)\,d\t g
=\int_{{\mathrm{End}}(\X)}\psi(x)\chi_{x}(w)\,dx\,,
\end{equation}
where
$
\psi(x)=sum(\Psi\Theta)(c^{-1}(x)) \cdot \Jac(x)
$
and $\Jac$ is the Jacobian of the inverse Cayley transform. Since $\psi\in C_c^\infty(\End^c(\X))$,  \eqref{theChc8GL} is a Fourier transform of $\psi$ evaluated at $\tau_{{\mathrm{End}}(\X)^*}(w)$. 
Hence $T(\Psi)$ is a function on $\W$ invariant under the action of $\F^\times$.
In terms of \eqref{TFofminimalorbit5}, \eqref{theChc2GL} is equal to
\[
\Chc(\Psi) =  \int_{{\mathrm{End}}(\X)}\psi(x)\delta(\det(x))\,dx
=\int_{\GL(\X)} sum(\Psi\Theta)(g)\delta(\det c(g))\,dg\,.
\]
Applying the formula $\delta(at)=|a|^{-1}\delta(t)$ to the above integral, we get \eqref{theChc3GL}. Through extension by zero, \eqref{theChc3GL} defines a distribution on $\wt\GL(X)$.

The claim \eqref{theChc5GL} follows from \cite[Proposition 5.27]{AubertPrzebinda_omega}.
From \cite[Proposition 5.27]{AubertPrzebinda_omega}, we have 
\[
\left|\Theta(\t g) \det(g-1)\right|=|\det g|^{\frac{1}{2}}\,.
\]
Hence, \eqref{theChc3GL} is equal to
\[
\int_{\GL(\X)}\Phi(g)|\det g|^{\frac{1}{2}}\delta(\det(g+1))\,dg\,,
\]
where $\Phi=sum(\Psi\epsilon)$.
Let $n=\dim_\F\X$. Then the last integral may be rewritten as
\begin{eqnarray}
\Chc(\Psi) & = & \int_{{\mathrm{End}}(\X)}\Phi(x)|\det x|^{\frac{1}{2}}\delta(\det(x+1))|\det(x)|^{-n}\,dx \nn \\
&=&\int_{{\mathrm{End}}(\X)}\Phi(x-1)|\det (x-1)|^{\frac{1}{2}-n}\delta(\det(x))\,dx\,.\label{theChc9GL} 
\end{eqnarray}
With reference to \eqref{theChc8GL}, we get
\[
\Chc(\Psi) =\int_{\W} \int_{{\mathrm{End}}(\X)}\psi(x) \chi_{x}(w)\,dx = \chc(\psi) = \int_{{\mathrm{End}}(\X)} \psi(x) \mu_{n-1}(x).
\]
The second equality is due to \eqref{TFofminimalorbit1} and \eqref{TFofminimalorbit2}.
By \eqref{TFofminimalorbit1}, \eqref{TFofminimalorbit2} and Remark \ref{TFofminimalorbitremark},
we have $\chc(x) = \delta(\det(x)) = \mu_{n-1}(x)$.

Notice that, with the appropriate normalization of measures,
\begin{equation}\label{gl1gln50}
\mu_{n-1}(\psi)=\int_{M_{n,n-1}(\F)} \psi^{\K}(0,z)\,dz\,,\ \ \ \psi^{\K}(x)=\int_{\K} \psi(kxk^{-1})\,dk\,.
\end{equation}
Hence, using the coordinates in the proof of Theorem \ref{TFofminimalorbit}, \eqref{theChc9GL} is equal to
\begin{eqnarray}\label{theChc10GL1}
\lefteqn{\Chc(\Psi) = \int_{{\mathrm{End}}(\X)}\Phi(x-1)|\det (x-1)|^{\frac{1}{2}-n}\,d\mu_{n-1}(x)} \\
&=&\int_{M_{n-1,n-1}(\F)}\int_{M_{1,n-1}(\F)} \Phi^\K(
\left(
\begin{array}{rrr}
-1 & z_1\\
0 & z_2-1
\end{array}
\right))
\left|
\det(
\left(
\begin{array}{rrr}
-1 & z_1\\
0 & z_2-1
\end{array}
\right))
\right|^{\frac{1}{2}-n}\,dz_1\,dz_2\nn\\
&=&\int_{M_{n-1,n-1}(\F)}\int_{M_{1,n-1}(\F)} \Phi^\K(
\left(
\begin{array}{rrr}
-1 & z_1\\
0 & z_2
\end{array}
\right))
\left|
\det(z_2)
\right|^{\frac{1}{2}-n}\,dz_1\,dz_2\nn\\
&=&\int_{M_{n-1,n-1}(\F)}\int_{M_{1,n-1}(\F)} \Phi^\K(
\left(
\begin{array}{rrr}
-1 & 0\\
0 & z_2
\end{array}
\right)
\left(
\begin{array}{rrr}
1 & -z_1\\
0 & 1
\end{array}
\right))
\left|
\det(z_2)
\right|^{-\frac{1}{2}}\,dz_1\left|
\det(z_2)
\right|^{1-n}\,dz_2\nn\\
&=&\int_{\GL_{n-1}(\F)}\int_{M_{1,n-1}(\F)} \Phi^\K(
\left(
\begin{array}{rrr}
-1 & 0\\
0 & g_2
\end{array}
\right)
\left(
\begin{array}{rrr}
1 & z_1\\
0 & 1
\end{array}
\right))
\left|\det(g_2)\right|^{-\frac{1}{2}}\,dz_1
\,dg_2\nn\,.
\end{eqnarray}
Since
\[
\left(
\begin{array}{rrr}
-1 & 0\\
0 & g_2
\end{array}
\right)
\left(
\begin{array}{rrr}
0 & z_1\\
0 & 0
\end{array}
\right)
\left(
\begin{array}{rrr}
-1 & 0\\
0 & g_2
\end{array}
\right)^{-1}
=\left(
\begin{array}{rrr}
0 & -z_1g_2^{-1}\\
0 & 0
\end{array}
\right)
\]
we see that
\[
\left|\det(g_2)\right|^{-\frac{1}{2}}
=\left|\Ad(
\left(
\begin{array}{rrr}
-1 & 0\\
0 & g_2
\end{array}
\right))_\n
\right|^{-\frac{1}{2}}\,,
\]
where
\[
\n=\left\{
\left(
\begin{array}{rrr}
0 & z_1\\
0 & 0
\end{array}
\right);\ z_1\in M_{1,n-1}(\F) \right\}\,.
\]
Therefore in the notation of \eqref{eqq11.1},
\begin{align*}
\Chc(\Psi) & = \int_{\GL_{n-1}(\F)} \Phi^{\GL_1(\F)\times\GL_{n-1}(\F)}(
\left(
\begin{array}{rrr}
-1 & 0\\
0 & g_2
\end{array}
\right))\,dg_2\\
& =\int_{\GL_{n-1}(\F)} \Phi^{\GL_1(\F)\times\GL_{n-1}(\F)}(
\left(
\begin{array}{rrr}
-1 & 0\\
0 & -g_2
\end{array}
\right))\,dg_2\\
& =\int_{\GL_{n-1}(\F)} (\Lambda_{-1}\Phi)^{\GL_1(\F)\times\GL_{n-1}(\F)}(
\left(
\begin{array}{rrr}
1 & 0\\
0 & g_2
\end{array}
\right))\,dg_2\\
& =\Ind_{\GL_1(\F)\times\GL_{n-1}(\F)}^{\GL_n(\F)}(\delta_1\times dg_2)(\Lambda_{-1}\Phi) \qquad \text{ (by \eqref{induceddistribution}.)} \\
& =\Ind_{\GL_1(\F)\times\GL_{n-1}(\F)}^{\GL_n(\F)}(\delta_1\times dg_2)(\Lambda_{-1}(sum(\Psi\epsilon))\,.
\end{align*}
This verifies \eqref{theChc6GL}.

Notice that the group $\GL(\X)$ is contained in its Lie algebra $\End(\X)$.
We see from \eqref{theChc9GL} that $\Chc$ coincides with the distribution $\delta\circ \det=\mu_{n-1}$ multiplied by the smooth function $|\det(x)|^{\frac{1}{2}}$ and translated by $-1$. Hence one obtains \eqref{theChc7GL} from Theorem \ref{TFofminimalorbitWF} via the substitution $x=g+1$.
\end{prf}
\subsection{\bf The definition of $\Chc_{h'}$.}\label{The definition of Chc}
\
Let $\G \cdot \G'\subseteq \Sp(\W)$ be an irreducible dual pair. 
Let $\H'\subseteq \G'$ be a Cartan subgroup with the split part $\A'\subseteq \H'$. Recall the dual pair $(\A'',\A')$ defined in Lemma \ref{question1}.

\begin{lem}\label{question2Chc}
For any $\Psi\in C_c^\infty(\wt{\A''{}^c})$, the distribution
\begin{equation}\label{theChc0GLChc}
T(\Psi)=\int_{\wt{\A''{}^c}} \Psi(\t g)T(\t g)\,d\t g\in \Ss^*(\W)
\end{equation}
is a function on $\W$, such that
\begin{equation}\label{theChcGLChc}
\int_{\A'\backslash\W_{\A'}}\left|\int_{\A''}\Psi(g)T(g)(w)\,dx\right| \,d\overset . w<\infty \,.
\end{equation}
The formula
\begin{equation}\label{theChc2GLChc}
\Chc(\Psi)=\int_{\A'\backslash\W_{\A'}} T(\Psi)(w)\,d(\A' w)
\qquad (\Psi\in C_c^\infty(\wt{\A''{}^c}))
\end{equation}
defines a distribution on $\wt{\A''{}^c}$ which coincides with a complex valued measure.
This measure extends by zero to $\wt{\A''}$ and defines a distribution, which shall be denoted by the same symbol, \eqref{theChc2GLChc}.
Furthermore,
\begin{equation}\label{theChc7GLChc}
\WF(\Chc) = \{(\t g,\tau_{\a''{}^*}(w));\ \t g\in \wt{\A''}\,,\ \tau_{\a''{}^*}(w)\ne 0\,, g(w)=-w\}\,,
\end{equation}
\end{lem}

\begin{prf}
This is immediate from \eqref{A''a''},  \eqref{A''a''I.1} and Theorems \ref{theChc} and \ref{theChcGL}.
\end{prf}

\begin{pro}\label{proposition1.8Chc}
Fix an element $\t h'\in \reg{\wt{\H'}}$. Then the intersection of the wave front set $\WF(\Chc)$ with the conormal bundle to the embedding
\begin{equation}\label{proposition 1.8.1.Chc}
\wt{\G}\ni \t g\to \t g\t h'\in \wt{\A''}
\end{equation}
is empty.
\end{pro}

\begin{prf}
The wave front set $\WF(\Chc)$ is equal to the set $S_{\A''}$ defined in \eqref{geometricII.1} and \eqref{geometricI.1}. The proposition follows from Lemmas \ref{geometricII} and \ref{geometricI}.
\end{prf}
Standard micro-local analysis, i.e. Theorem \ref{Theorem 8.2.4}, together with Proposition \ref{proposition1.8}
justify the following definition.

\begin{defi}\label{Definition 1.9.Chc}
For $\t h'\in \reg{\wt{\H'}}$ define $\Chc_{\t h'}$ to be the pullback of the distribution $\Chc$ to  $\wt{\G}$ via the embedding \eqref{proposition 1.8.1.Chc}.
Formally
\begin{equation}\label{question2.2agChc}
\Chc_{\t h'}(\t g)=\int_{\A'\backslash \W_{\A'}} T(\t g\t h')(w)\,d(\A' w) \qquad (x\in \g)\,.
\end{equation}
\end{defi}

\begin{lem}\label{Lemma 1.10.Chc}
For any $x'\in \reg{\h'}$ 
\[
\WF(\Chc_{\t h'})\subseteq\{(\t g,\tau_{\g^*}(w));\ gh'(w)=-w,\ \t  g \in\wt{\G},\ w\in \W\}\,.
\]
\end{lem}
\begin{prf}
This follows from Definition \ref{Definition 1.9.Chc}, \eqref{theChc7GLChc} and Theorem \ref{Theorem 8.2.4}. 
\end{prf}
\section{\bf The $\chc$ is the lowest term in the asymptotic expansion of $\Chc$.}\label{asymptotic expansion}

In this section we verify the following theorem which shows that $\chc$ is the `the lowest term in the asymptotic expansion' of $\Chc$.
\begin{thm}\label{asymptoticexpansion}
In terms of Definitions \ref{Definition 1.9.} and \ref{Definition 1.9.Chc}, for $x\in\g$ and 
$x'\in\reg{\h'}$,
\[
\underset{t\to 0}{\lim}\, t^{\dim \W}\int_{\A'\backslash\W_{\A'}} T(-\t c(t^2x)\t c(t^2x'))(w)\,d(\A'w)=\Theta(\wt{-1})\int_{\A'\backslash\W_{\A'}} \chi_{x+x'}(w)\,d(\A' w)\,.
\]
\end{thm}
\begin{prf}
Here we follow the proof of \cite[Theorem 2.13]{PrzebindaCauchy}. Recall that in terms of rational functions,
\[
c(-c(x)c(x'))=c(c(x)c(x'))^{-1}=((x'-1)(x+x')^{-1}(x-1)+1)^{-1}\,.
\]
Hence,
\[
c(-c(t^2x)c(t^2x'))=t^2((t^2x'-1)(x+x')^{-1}(t^2x-1)+t^2)^{-1}\,.
\]
Therefore
\begin{multline*}
\int_{\A'\backslash\W_{\A'}} \chi_{ c(-c(t^2x)c(t^2x'))}(w)\,d(\A' w)
=\int_{\A'\backslash\W_{\A'}} \chi_{((t^2x'-1)(x+x')^{-1}(t^2x-1)+t^2)^{-1}}(tw)\,d(\A' w) \\
=
t^{-\dim \W}\int_{\A'\backslash\W_{\A'}} \chi_{((t^2x'-1)(x+x')^{-1}(t^2x-1)+t^2)^{-1}}(w)\,d(\A' w)
\end{multline*}
Hence,
\begin{multline*}
 t^{\dim \W}\int_{\A'\backslash\W_{\A'}} T(-\t c(t^2x)\t c(t^2x'))(w)\,d(\A'w)\\
 = \Theta(-\t c(t^2x)\t c(t^2x'))\int_{\A'\backslash\W_{\A'}} \chi_{((t^2x'-1)(x+x')^{-1}(t^2x-1)+t^2)^{-1}}(w)\,d(\A' w)\,,
\end{multline*}
which has the desired limit if $t\to 0$.
\end{prf}
\section{\bf Conjectures and motivation}\label{Conjectures and motivation}
\ 
Let $\Pi'$ be an irreducible admissible representation of $\wt{\G'}$ which occurs in Howe's correspondence for the pair $(\wt\G, \wt{\G'})$ and let $\Pi_1$ be the corresponding maximal Howe's quotient  representation of $\wt\G$. Let $\Theta_{\Pi'}$ denote the distribution character of $\Pi'$. 

We consider the product $Z'\G'{}^0$ where $Z'$ is the center of $\G'$ and $\G'{}^0$ is the Zariski identity component of $\G'$. We recall Section \ref{The Weyl Harish-Chandra integration formula} that $Z'\G'{}^0=\G'$, unless $\G'$ is an even orthogonal group.

\begin{conj}
For any $\Psi\in C_c(\wt\G)$
\begin{equation}\label{Conjectureforcharacters1}
\int_{\reg{\H'}}\left|\Theta_{\Pi'}(\t h')D(h')\Chc_{\t h'}(\Psi)\right|\,d h' <\infty\,.
\end{equation}
Assuming \eqref{Conjectureforcharacters1}, in terms of  the Weyl - Harish-Chandra integration formula for $\G'$ (see \eqref{Weyl - Harish-Chandra integration formula on G}), define
\begin{equation}\label{Conjectureforcharacters2}
\Theta_{\Pi'}'(\Psi)=
C_{\Pi'}\sum_{\H'}\frac{1}{|W(\H')|}\int_{\reg{\H'}}\check\Theta_{\Pi'}(\t h')|D(h')|\frac{1}{\mu(\H'/\A')}\Chc_{\t h'}(\Psi)\,d\t h'\,,
\end{equation}
where $\check\Theta_{\Pi'}(\t h')=\Theta_{\Pi'}(\t h'^{-1})$ and
\[
C_{\Pi'}=(\text{the central character of $\Pi'$ evaluated at $\wt{-1}$})^{-1}\cdot\Theta(\wt{-1})\,.
\]
If the character $\Theta_{\Pi'}$ is supported in $Z'\G'{}^0$, then, as distributions, 
\begin{equation}\label{Conjectureforcharacters3}
\Theta_{\Pi'}'=\Theta_{\Pi_1}\,.
\end{equation}
\end{conj}

Here is a heuristic reason why \eqref{Conjectureforcharacters3} should hold. The right hand side of 
\eqref{Conjectureforcharacters2}, divided by $C_{\Pi'}$, is equal to
\begin{multline*}
\sum_{\H'}\frac{1}{|W(\H')|}\int_{\reg{\H'}}\check\Theta_{\Pi'}(\t h')|D(h')|\frac{1}{\mu(\H'/\A')}\Chc_{\t h'}(\Psi)\,d h'\\
=\sum_{\H'}\frac{1}{|W(\H')|}\int_{\reg{\H'}}\check\Theta_{\Pi'}(\t h')|D(h')|\frac{1}{\mu(\H'/\A')}
\int_{\A'\backslash \W_{\A'}}\int_{\wt{\G}}\Psi(\t g) T(\t g\t h')(w)\,d\t g\,d(\A' w)\\
=\sum_{\H'}\frac{1}{|W(\H')|}\int_{\reg{\H'}}\check\Theta_{\Pi'}(\t h')|D(h')|
\int_{\H'\backslash \W_{\A'}}\int_{\wt{\G}}\Psi(\t g) T(\t g\t h')(w)\,d\t g\,d(\H' w)
\,d h'\\
=\int_{\reg{\G'}}\check\Theta_{\Pi'}(\t g')
\int_{\G'\backslash \W_{\A'}}\int_{\wt{\G}}\Psi(\t g) T(\t g\t g')(w)\,d\t g\,d(\G' w)
\,d\t g'\,.
\end{multline*}
Suppose $\G'$ is compact. Then the above is equal to
\[
\int_{\reg{\G'}}\check\Theta_{\Pi'}(\t g')
\int_{\W}\int_{\wt{\G}}\Psi(\t g) T(\t g\t g')(w)\,d\t g\,dw
\,d\t g'\,.
\]
On the other hand, formally, for any $\t s\in\wt{\Sp}(\W)$,
\begin{multline*}
\Theta(\wt{-1})\int_{\W}T(\t s)(w)\,dw = \int_{\W}\Theta(\wt{-1}) \chi_0(w)T(\t s)(w)\,dw
= \int_{\W}\Theta(\wt{-1}) \chi_0(-w)T(\t s)(w)\,dw\\
=\int_{\W}T(\wt{-1})(-w)T(\t s)(w)\,dw = T(\wt{-1})\natural T(\t s)(0)
=T(\wt{-1}\t s)(0)=\Theta(\wt{-1}\t s)\,.
\end{multline*}
Let $\zeta$ be the central character of $\Pi'$ evaluated at $\wt{-1}$.
Hence the right hand side of \eqref{Conjectureforcharacters2} is equal to
\begin{multline*}
\zeta^{-1}\int_{\wt{\G'}}\check\Theta_{\Pi'}(\t g')\int_{\wt{\G}}\Psi(\t g)\Theta(\wt{-1}\t g\t g'), d\t g\,d\t g'\\
=\zeta^{-1}\int_{\wt{\G'}}\check\Theta_{\Pi'}(\wt{-1}^{-1}\t g')\int_{\wt{\G}}\Psi(\t g)\Theta(\t g\t g')\, d\t g\,d\t g'\\
=\int_{\wt{\G'}}\check\Theta_{\Pi'}(\t g')\int_{\wt{\G}}\Psi(\t g)\Theta(\t g\t g')\,d\t g\,d\t g'=\int_{\wt{\G}}\Psi(\t g)\Theta_{\Pi}(\t g)\,d\t g\,.
\end{multline*}

\appendix

\section{\bf The wave front set}\label{appenA}
\setcounter{thh}{0}
\renewcommand{\thethh}{A.\fontindex{thh}}
\setcounter{equation}{0}
\renewcommand{\theequation}{A.\fontindex{equation}}

In this appendix  we gather some facts from \cite[chapter VIII]{Hormander} that we need in the case when the real numbers are replaced by a $p$-adic field $\F$, i.e. a finite extension of $\Qp$. 

Let $\fo_F$ be the ring of integers in $\F$ and let $\mathfrak p_F \subseteq \fo_F$ the unique maximal proper ideal.
Let $\chi \colon \F \rightarrow {\mathbb{C}}^\times$ be a character of the additive group $\F$ such that the kernel of $\chi$ is equal to $\fo_F$. Denote by $\Ss(U)$  the Schwartz-Bruhat space on an open subset $U\subseteq\F^n$, i.e. the space of complex-valued locally constant functions with compact support on $U$. We identify $\F^n$ with its linear dual $(\F^n)^*$ via the following symmetric bilinear form
\[
x\cdot y=\sum_{j=1}^nx_jy_j=x y^t \,,
\]
where $x$ and $y$ are row vectors in $\F^n$ with the indicated entries.
We normalize the Haar measure $dx=dx_1 dx_2 \ldots dx_n$ on $\F^n$ so that
for $\phi\in\Ss(\F^n)$ the Fourier transform and its inverse are given by
\begin{equation}\label{Fourier transform F}
\cF \phi(\xi)=\int_{\F^n} \phi(x)\chi(-x\cdot \xi)\,dx\qquad (\xi\in \F^n)
\end{equation}
and
\begin{equation}\label{Fourier inversion formula F}
\phi(x)=\int_{\F^n}\cF \phi(\xi)\chi(x\cdot \xi)\,d\xi \qquad (x\in \F^n)\,,
\end{equation}
see  \cite[Corollary 1, page 107]{Weil_Basic}.

Denote by $| \  |$ the  absolute value on $\F$, as in \cite[page 4]{Weil_Basic}, and let
\[
|x|=\max\{|x_1|, \ldots , |x_1|\} \qquad (x=(x_1, \ldots , x_n)\in \F^n)\,.
\]
Fix a subgroup $\Lambda\subseteq\F^\times$ of finite index. We will assume that
\begin{equation}\label{additional assumption on Lambda}
\varpi\in\Lambda\,.
\end{equation}
Then $\Lambda$ is the direct product of the subgroup of $\langle \varpi\rangle\subseteq \F^\times$ generated by $\varpi$ and a subgroup of $\Lambda\cap \fo_F^\times\subseteq \fo_F^\times$ of finite index,
\[
\Lambda=\langle \varpi\rangle(\Lambda\cap \fo_F^\times)\,.
\]

A subset of $\F^n\setminus 0$ is called a cone if it is closed under the multiplication by $\Lambda$. (A more correct name would be a $\Lambda$ - cone. Since $\Lambda$ is fixed, we omit it.) For $n=1$,  there are finitely many elements $x_1, x_2,\ldots , x_l\in \fo_F^\times$ such that
$\F\setminus 0=\cup_{j=1}^l \Lambda x_j$. 

More generally, let $S^{n-1}=\{(x_1, x_2,\ldots , x_n)\in\F^n;\ \max_i|x_i|=1\}$ be the unit sphere. 
This is an open and compact subset of $\F^n$. (See the proof of Theorem \ref{TFofminimalorbit} for more information on the structure of $S^{n-1}$.) Then for any cone $\Gamma\subseteq \F^n\setminus 0$, we have
\begin{equation}\label{structure of cones}
\Gamma=\Lambda\,(\Gamma\cap S^{n-1})
\end{equation}
(see \cite[Lemma 2.5]{Heifetz}). Indeed, if $(x_1, x_2,\ldots , x_n)\in \Gamma$ and $|x_j|=\max_i|x_i|$ with $|x_j|=|\varpi^m|$, then
$\varpi^{-m}(x_1, x_2,\ldots , x_n)\in \Gamma$, because $\varpi\in\Lambda$ by \eqref{additional assumption on Lambda}.

We see from \eqref{structure of cones} that 
$\Gamma$ is open (respectively closed) if and only if $\Gamma\cap S^{n-1}$ is open (respectively closed). 

Moreover $\Gamma\cap S^{n-1} = \bigcup_{i\in I} S_i$ is a union of orbits $S_i$ under the action of the group $\Lambda\cap \fo_\F^\times$. Thus
\[
\Gamma=\bigcup_{i\in I} \langle \varpi\rangle S_i\,.
\]
For a subset $A \subseteq \F^n\setminus 0$ we define the asymptotic cone $\rAC(A)$ of $A$ to be the complement of the union of all the open cones whose intersection with $A$ is bounded. 
Equivalently, a point $x\in \F^n\setminus 0$ belongs to $\rAC(A)$ if and only if every open cone containing $x$ intersects $A$ in an unbounded set.  

Since the unit sphere $S^{n-1}$ is compact, any covering of $\F^n\setminus 0$ by open cones contains a finite subcovering. Therefore, if $\rAC(A)=\emptyset$ then $A$ is bounded. Conversely, if $A$ is bounded, then $\rAC(A) =\emptyset$. Thus, $\rAC(A)=\emptyset$ if and only if $A$ is bounded.

Let $\Ee^*(\F^n)\subseteq \Ss^*(\F^n)$ denote the subspace of the compactly supported distributions. 
For $ v \in \Ee^*(\F^n)$, it operates on the space $C^\infty(\F^n)$ of the locally constant functions in the following way:
\[
v(\phi)=v(\psi\phi) \qquad (\phi\in C^\infty(\F^n), \psi\in \Ss^*(\F^n), \psi=1\ \text{on}\ \supp v)\,.
\]
In particular its Fourier transform
\begin{equation}\label{Formal Fourier transform of v}
\cF v(\xi)=v(\chi(-\ \cdot\xi)) \qquad (\xi\in \F^n)
\end{equation}
is well defined. A straightforward computation shows that  $\cF v\in C^\infty(\F^n)$. Suppose $\supp \cF v$ is bounded. Then $v=\cF^{-1} \cF v$ is a smooth function. This is a special case of
a non-Archimedean version of the Paley-Wiener Theorem which  says that a distribution $v\in\Ee^*(\F^n)$ coincides with the locally constant function \eqref{Formal Fourier transform of v} times the Haar measure if its Fourier transform is compactly supported, i.e. $\cF v\in \Ee^*(\F^n)$. 

Therefore, if $v\in\Ee^*(\F^n)$ is not a smooth function, then $\supp \cF v$ is unbounded. Hence $\rAC(\supp \cF v)\ne \emptyset$. Following \cite[\S 8.1]{Hormander}, we define $\Sigma(v)=\rAC(\supp \cF v)$.  Clearly
\begin{equation}\label{inclusion for sum}
\Sigma(v_1+v_2)\subseteq \Sigma(v_1) \cup\Sigma(v_2) \qquad (v_1, v_2\in \Ee^*(\F^n))\,.
\end{equation}

\begin{lem}\label{A and AplusB}
For any subset $A\subseteq \F^n\setminus 0$ and any bounded set $B\subseteq \F^n$, 
\[
\rAC(A)\supseteq \rAC(A+B)\,.
\]
\end{lem}
\begin{prf}
We need to show that for any open cone $V\subseteq \F^n\setminus 0$, if $V\cap A$ is bounded then $V\cap (A+B)$ is bounded. 

Let $W=V\cap S^{n-1}$, so that $V=\Lambda W$, see \eqref{structure of cones}. Since  $W$ is open, there is $M>0$ such that
\begin{equation}\label{maininclusion}
W-\lambda^{-1} B\subseteq W \qquad (\lambda\in \Lambda, |\lambda|>M)\,.
\end{equation}
Also, since  $V\cap A$ is bounded and since $W$ is bounded away from $0$, there in $N>0$ such that 
\begin{equation}\label{maininclusion1}
\lambda W\cap A=\emptyset \qquad (\lambda\in \Lambda, |\lambda|>N)\,.
\end{equation}
Hence
\[
\lambda W\cap (A+B)=\emptyset \qquad (\lambda\in \Lambda, |\lambda|>\max\{M, N\})\,.
\]
Indeed, if not then there are $a\in A$, $b\in B$, $w\in W$ and $\lambda$ as above such that
\[
\lambda w=a+b\,.
\]
But then
\[
a=\lambda(w-\lambda^{-1}b)\in \lambda(W-\lambda^{-1}B)\subseteq \lambda W\,,
\]
which contradicts \eqref{maininclusion1}. 
\end{prf}

\begin{cor}\label{HoremanderLemma8.1.1}\cite[Lemma 8.1.1]{Hormander}
For $\phi\in \Ss(\F^n)$ and $v\in \Ee^*(\F^n)$, $\Sigma(\phi v)\subseteq\Sigma(v)$.
\end{cor}

\begin{defi}\label{Hormender'sWF}\cite[Lemma 8.1.2]{Hormander}
For $u\in \Ss^*(\F^n)$ define
\[
\Sigma_x(u)=\bigcap_{\phi\in \Ss(\F^n),\ \phi(x)\ne 0} \Sigma(\phi u) \qquad (x\in \F^n)
\]
and the wave front set of $u$,
\[
\WF(u)=\{(x,\xi)\in \F^n\times (\F^n\setminus 0);\ \xi\in \Sigma_x(u)\} \,.
\]
(Note that if $\Sigma_x(u)\ne \emptyset$, then $x\in\supp(u)$.)
\end{defi}

\begin{lem}\label{closertoh}
Fix a subset $A\subseteq\F^n\setminus 0$ and an open compact set $W\subseteq\F^n\setminus 0$. 
For $N>0$ let $\Lambda_{>N}=\{\lambda\in\Lambda;\ |\lambda|>N\}$. \footnote{Loke: We recall that $\Lambda$ is a subgroup of finite index in $\F^\times$.}
Then there is $N$ such that
\[
(\Lambda_{>N}W)\cap A=\emptyset 
\]
if and only if
\[
\Lambda W\cap A \ \ \ \text{is bounded}\,.
\]
\end{lem}
\begin{prf}
Set 
\[
(\Lambda W)_{>M}=\{x\in \Lambda W;\ |x|>M\}\,.
\]
Then
\[
\Lambda_{>N}W\supseteq (\Lambda W)_{>N\cdot\max\{|\xi|;\ \xi\in W\}}
\]
and
\[
(\Lambda W)_{>M}\supseteq \Lambda_{>M/\min\{|\xi|;\ \xi\in W\}}W\,.
\]
This implies the claim.
\end{prf}

\begin{defi}\label{Heifetzsmooth}\cite[pages 288-289]{Heifetz}
A distribution $u\in \Ss^*(\F^n)$ is smooth at a point $(x_0,\xi_0)\in \F^n\times(\F^n\setminus 0)$ if and only if there are open compact subsets $U\subseteq \F^n$ and $W\subseteq \F^n\setminus 0$ such that $x_0\in U$, $\xi_0\in W$ and for any $\phi\in \Ss(U)$ there is $N>0$ such that
\[
(\Lambda_{>N} W)\cap \supp \cF(\phi u)=\emptyset \,.
\]
\end{defi}

\begin{lem}\label{HormenderWFisHeifetzWF}
For any $u\in \Ss^*(\F^n)$, the wave front set of $u$ is equal to the complement of the set of all the smooth points in $\F^n\times(\F^n\setminus 0)$ of $u$.
\end{lem}
\begin{prf}
We see from Lemma \ref{closertoh} that a $(x_0,\xi_0)$ is a smooth point for $u$ if and only if there are open compact subsets $U\subseteq \F^n$ and $W\subseteq \F^n\setminus 0$ such that $x_0\in U$, $\xi_0\in W$ and for any $\phi\in \Ss(U)$ 
\[
\Lambda W\cap \supp\cF(\phi u)
\]
is bounded. Equivalently,  there is an open compact subset $U\subseteq \F^n$ and open cone $V\subseteq \F^n\setminus 0$ such that $x_0\in U$, $\xi_0\in V$ and for any $\phi\in \Ss(U)$ 
\[
V\cap \supp\cF(\phi u)
\]
is bounded.
Equivalently, there is open compact subset $U\subseteq \F^n$ containing $x_0$ such that
\[
\xi_0 \notin \Sigma(\phi u) \qquad (\phi\in \Ss(U))\,.
\]

Suppose $(x_0,\xi_0)$ is a smooth point, then the above implies $\xi_0 \notin \Sigma_{x_0} (u)$ so $(x_0,\xi_0) \notin \WF(u)$.

Conversely $(x_0,\xi_0) \notin \WF(u)$ so $\xi_0 \notin \Sigma_{x_0}(u)$. 
By definition, there is a $\phi\in \Ss(U)$ with $\phi(x_0)\ne 0$, such that $\xi_0\notin \Sigma(\phi u)$. Let $U_0\subseteq U$ be an open set where $\phi$ is constant equal to $\phi(x_0)$. By Corollary \ref{HoremanderLemma8.1.1}, 
\[
\Sigma(\phi_0\phi u)\subseteq \Sigma(\phi u) \qquad (\phi_0\in S(U_0))\,.
\]
Therefore
\[
\xi_0\notin \Sigma(\phi_0 u) \qquad (\phi_0\in \Ss(U_0))\,.
\]
This shows that $(x_0,\xi_0)$ is a smooth point for $u$.
\end{prf}

\begin{cor}\label{WFisClosed}\cite[(8.1.9)]{Hormander}
For any $u\in \Ss^*(\F^n)$, the subset $\WF(u)\subseteq \F^n\times (\F^n\setminus 0)$ is closed. Moreover, 
\[
\WF(\phi u)\subseteq \WF(u) \qquad (\phi\in C^\infty(\F^n))\,.
\]
In particular, if $\phi$ has no zeros, then
\[
\WF(\phi u)= \WF(u)\,.
\]
\end{cor}

\begin{cor}\label{projectionWF}\cite[Proposition 8.1.3]{Hormander}
For any $v\in \Ee^*(\F^n)$ the projection of $\WF(v)$ onto the second component is equal to $\Sigma(v)$. 
\end{cor}
\begin{prf}
Let 
\[
P(\WF(v))=\bigcup_{x\in \F^n}\Sigma_x(v)
\]
denote this projection. Clearly
\begin{equation}\label{projectionWF1}
P(\WF(v))\subseteq \Sigma(v)\,.
\end{equation}
Since $\WF(v)$ is closed and $\supp v$ is compact, the set $\WF(v)\cap (\F^n\times S^{n-1})$ is compact. Hence $P(\WF(v))\cap S^{n-1}$ is compact. Since, by \eqref{structure of cones}, $P(\WF(v))=\Lambda(P(\WF(v))\cap S^{n-1})$, the cone $P(\WF(v))$ is closed. 

In order to show the equality
\[
P(\WF(v))=\Sigma(v)
\]
it will suffice to check that for any open cone $V$ containing $P(\WF(v))$, we have $\Sigma(v) \subseteq V$. Let $V$ be such a cone. 
Suppose $\Sigma_x(v)\ne \emptyset$. Then $x\in \supp v$. Then by \eqref{projectionWF1}, $\Sigma_x(v)\subseteq V$. Hence, by Definition \ref{Hormender'sWF}, there is an open neighborhood $U_x\subseteq \F^n$ such that
\[
\Sigma(\phi v)\subseteq V \qquad (\phi\in \Ss(U_x))\,.
\]
We can cover $\supp v$ by a finite number of such $U_{x_j}$ and choose $\phi_j\in \Ss(U_{x_j})$ with $\sum_j\phi_j=1$ on $\supp v$. Then, by \eqref{inclusion for sum},
\[
\Sigma(v)=\Sigma(\sum_j \phi_j v)\subseteq\bigcup_j \Sigma(\phi_j v)\subseteq V\,.
\]
\end{prf}

\begin{lem}\label{sigma0inacsupp}\cite[Lemma 8.1.7]{Hormander}
For any $u\in \Ss^*(\F^n)$
\[
\Sigma_0(u)\subseteq \rAC(\supp \cF(u))\,.
\]
\end{lem}
\begin{prf}
Let $\phi\in\Ss(\F^n)$ be such that $\phi(0)\ne 0$. Then, by Definition \ref{Hormender'sWF}, 
\[
\Sigma_0(u)\subseteq\Sigma(\phi u)\,.
\]
We also have
\begin{equation*}
\Sigma(\phi u) = \rAC(\supp\cF(\phi u))\subseteq \rAC(\supp\cF(\phi)+\supp\cF(u))
\subseteq  \rAC(\supp\cF(u))\,,
\end{equation*}
where last inclusion follows from Lemma \ref{A and AplusB}, because $\cF(\phi) \in \Ss(\F^n)$ has bounded support.
\end{prf}

\begin{defi}\label{homogdistr}
Let $\chi_\Lambda:\Lambda\to \C^\times$ is a group homomorphism. A distribution $u\in \Ss^*(\F^n)$ is called homogeneous of degree $\chi_\Lambda$ if
\[
u(\phi_\lambda)=\chi_\Lambda(\lambda)u(\phi)\qquad (\lambda\in\Lambda, \phi\in \Ss(\F^n))\,,
\]
where
\[
\phi_\lambda(x)=|\lambda|^{-n}\phi(\lambda^{-1}x)\,.
\]
(An example of such a homogeneous distribution is the Haar measure on $\F^n$. The homogeneous distributions are described completely in \cite[chapter II, section 2.3, page 138]{GelfandGraevPS}.)
\end{defi}
\begin{lem}\label{homogdistr111}
If $u\in \Ss^*(\F^n)$ is a homogeneous distribution, then $\Sigma_0(u)=\supp \cF u$.
\end{lem}
\begin{proof}
Since $u\in \Ss^*(\F^n)$ is homogeneous of degree $\chi_\Lambda$, $\cF u$ is homogeneous of degree $\chi_\Lambda^{-1}|\cdot |^{-n}$. The support of a homogeneous distribution is a cone so it is equal to its asymptotic cone. Lemma \ref{sigma0inacsupp} gives
\[
\Sigma_0(u)\subseteq \rAC(\supp \cF(u)) = \supp \cF u\,.
\]
Suppose $0\ne \xi\notin \Sigma_0(u)$. We will show that $\xi \not \in \supp \cF u$. This will prove $\Sigma_0(u)\supseteq \supp \cF u$.

Now $0 \ne \xi\notin \Sigma_0(u)$ means that
$(0,\xi)$ is a smooth point of the distribution $u$.
By definition, there is an open compact set $U$ containing $0$, an open compact set $W$ containing $\xi$ and a number $N>0$ such that
\[
\cF(\phi u)(\lambda w)=0 \qquad (\lambda\in \Lambda\,,\ |\lambda|>N\,,\ w\in W)\,.
\]
Let 
\[
L_w\psi(x)=\psi(x-w)\,,\ R\psi(x)=\psi(-x)\,,\ D_{\lambda^{-1}}\psi(x)=|\lambda|^{n}\psi(\lambda x)\,.
\]
Then a simple computation shows that
\[
\cF(\phi u)(\lambda w)=\chi_\Lambda(\lambda)^{-1}\cF u(L_w R D_{\lambda^{-1}} \cF\phi)\,.
\]
Hence
\begin{equation} \label{eqFuLw}
\cF u(L_w R D_{\lambda^{-1}} \cF\phi)=0 \qquad (\lambda\in \Lambda\,,\ |\lambda|>N\,,\ w\in W)\,. 
\end{equation}
For $\epsilon>0$ let $B_\epsilon\subseteq \F^n$ denote the closed ball of radius $\epsilon$ centered at $0$. Let $1_{B_\epsilon+w}$ be the indicator function of $B_\epsilon+w$. We set
\[
\phi= \cF^{-1}D_{\lambda}RL_{-w}1_{B_\epsilon+w}=\cF^{-1}D_{\lambda}R1_{B_\epsilon}
=|\lambda|^{-n}\cF^{-1}D_{\lambda}1_{B_\epsilon}=\cF^{-1}1_{B_{|\lambda|\epsilon}}
=C(|\lambda|)1_{B_{\frac{1}{|\lambda|\epsilon}}}\,,
\]
where $C(|\lambda|)$ is a non-zero constant. For $|\lambda|$ large enough $\phi\in \Ss(U)$. This implies
\[
1_{B_\epsilon+w}=L_w R D_{\lambda^{-1}} \cF\phi\,.
\] 
By \eqref{eqFuLw} we have
\begin{equation} \label{eqcFu}
\cF u(1_{B_\epsilon+w})=0\,,
\end{equation}

For any $w \in W$, we can find $\epsilon>0$ such that $B_\epsilon+w\subseteq W$. Then \eqref{eqcFu} implies that 
\[
W\cap \supp \cF u=\emptyset.
\]
In particular $\xi \not \in \supp \cF$ as required. 
\end{proof}
The following theorem does not seem to appear in the literature, though the proof is a straightforward adaptation of the argument used by H\"ormander.

\begin{thm}\label{a1}\cite[Theorem 8.1.8]{Hormander}
Suppose $u\in\Ss^*(F^n)$ is homogeneous. Then
\begin{eqnarray}
\begin{array}{llllll}
(x,\xi)\in \WF(u) &\iff &(\xi,-x)\in \WF(\cF u)  &(x\ne 0\ \text{and}\ \xi\ne 0)\,,\nn\\
x\in\supp u &\iff  &(0,-x)\in \WF(\cF u) &(x\ne 0)\,,\nn\\
\xi\in\supp \cF u &\iff  &(0,\xi)\in \WF(u) &(\xi\ne 0)\nn\,.
\end{array}
\end{eqnarray}
\end{thm}

\begin{prf}
The Fourier transform of a homogeneous distribution is homogeneous and the composition of two Fourier transforms is the reflection. Hence the last two statements are equivalent. Moreover the last one coincides with the equality of Lemma \ref{homogdistr111}. Thus it'll suffice to verify the first one.

Since the composition of two Fourier transforms is the reflection, it'll suffice to show that for any $x_0\ne 0$ and $\xi_0\ne 0$, 
\begin{equation}\label{a1.implication}
(x_0,\xi_0)\notin \WF(u) \implies (\xi_0,-x_0)\notin \WF( \cF u)\,.
\end{equation}

In order to simplify notation we shall follow Gelfand and write
\[
u(\phi)=\int_{\F^n} u(x) \phi(x)\,dx
\]
for the value of the distribution $u$ on a test function $\phi$. Choose $\phi\in \Ss(\F^n)$ equal to $1$ in a neighborhood of $\xi_0$ and $\psi\in \Ss(\F^n)$ equal to $1$ in a neighborhood of $x_0$ so that
\begin{equation}\label{a1.4}
(\supp \psi \times \supp \phi)\cap \WF(u)=\emptyset\,.
\end{equation}
Let $v=\phi\cF u$. We need to estimate $\cF v(-\lambda x)$ in a $\Lambda$-conic neighborhood of $-x_0$.
Then for $\lambda\in\Lambda$ and $x\in \F^n$
\begin{eqnarray}\label{a1.5}
\cF v(-\lambda x)&=&\cF\phi*(\cF^2 u)(-\lambda x)=\int_{\F^n} \cF\phi(-\lambda x-y)(\cF^2 u)(y)\,dy\\
&=&\int_{\F^n} \cF\phi(-\lambda x+y) u(y)\,dy=\chi_\Lambda(\lambda)|\lambda|^{n}\int \cF\phi(\lambda(y-x))u(y)\,dy\nn\\
&=&\chi_\Lambda(\lambda)|\lambda|^{n}\int \cF\phi(\lambda(y-x))(\psi u)(y)\,dy\nn\\
& & {} + \chi_\Lambda(\lambda)|\lambda|^{n}\int \cF\phi(\lambda(y-x))((1-\psi) u)(y)\,dy\nn\,.
\end{eqnarray}

Let $r>0$ be such that $(1-\psi)(y)\ne 0$ implies $|y-x_0|>2r$. Then $|x-x_0|<r$ implies $|y-x|>r$. Thus there is $N>0$ such that 
\[
\cF\phi(\lambda(y-x))=0\qquad ((1-\psi)(y)\ne 0,\ |x-x_0|<r,\ |\lambda|>N)\,.
\]
Hence,
\begin{equation}\label{a1.6}
\cF\phi(\lambda(y-x))(1-\psi)(y)=0\qquad (|x-x_0|<r,\ |\lambda|>N)\,.
\end{equation}
Also, by the adjoint property of Fourier transform, 
\begin{equation}\label{a1.7}
\int \cF\phi(\lambda (y-x))(\psi u)(y)\,dy=\int \cF(\psi u)(\lambda \xi)\phi(\xi)\chi(\lambda x\cdot \xi)\,d\xi\,.
\end{equation}
The condition \eqref{a1.4} implies that there is $M>0$ such that
\[
\cF(\psi u)(\lambda \xi)=0\qquad (|\lambda|>M,\ \phi(\xi)\ne 0)\,.
\]
Thus \eqref{a1.7} is zero for $|\lambda|>M$ independently of $x$.
Therefore \eqref{a1.5} is zero if $|x-x_0|<r$ and $|\lambda|>\max\{M, N\}$. This verifies \eqref{a1.implication}.
\end{prf}

\begin{thm}\label{Theorem 8.1.5Hormander}\cite[Theorem 8.1.5]{Hormander}
Let $V\subseteq \F^n$ be a subspace and let $V^\perp\subseteq \F^n$ be the orthogonal complement, so that $\F^n=V\oplus V^\perp$. Let $u_V\in \Ss^*(V)$ be a multiple of the Haar measure $dx_V$ by a nonzero smooth function. Denote by $\delta\in \Ss^*(V^\perp)$ the Dirac delta at the origin. Then
\[
\WF(u_V\otimes \delta)=\supp u_V\times V^\perp\setminus 0\,.
\]
\end{thm}
\begin{prf}
We see from Corollary \ref{WFisClosed} that we may assume $u_V=dx_V$. 
Then the distribution $u_V\otimes \delta$ is homogeneous, so Theorem \ref{a1} implies 
\[
\Sigma_0(u_V\otimes \delta)=\supp\cF(u_V\otimes \delta)\setminus 0=\{0\}\times V^\perp\setminus 0\,.
\]
However, it is clear from Definition \ref{Hormender'sWF} that
\begin{eqnarray*}
\Sigma_x(u_V\otimes \delta)&=&\Sigma_0(u_V\otimes \delta)\qquad (x\in \supp u_V)\,,\\
\Sigma_x(u_V\otimes \delta)&=&\emptyset\qquad (x \notin \supp u_V)\,.
\end{eqnarray*}
\end{prf}

\begin{lem}\label{Theorem 7.7.1Hormander}\cite[Theorem 7.7.1]{Hormander}, \cite[Proposition 1.1]{Heifetz}
Let $U\subseteq \F^n$ be an open compact subset and let $f:U\to\F$ be a differentiable function such that 
\[
f(x_0+x)=f(x_0)+f'(x_0)\cdot x+R(x_0,x)(x)\cdot x \qquad (x_0,\ x_0+x\in U)\,,
\]
where $f'(x_0)\in\F^n$ is the gradient of $f$ at $x_0$ and $R(x_0,x):\F^n\to\F^n$ is a linear function with 
\[
B:=\max_{x_0, x_0+x\in U, |y|=1}|R(x_0, x)(y)\cdot y|<\infty\,.
\]
We also assume that
\[
\delta:=\min_{x_0\in U}|f'(x_0)|>0\,.
\]
Let $\phi\in \Ss(U)$ and
let $m_0\in \Bbb Z$ be the minimum of the $m\in \Bbb Z$ such that there is a finite disjointed covering
\[
U=\bigsqcup_k \left(x_k+B_{q^{-m}}\right)
\]
and $\phi$ is constant on each $x_k+B_{q^{-m}}$.
(The covering exists because $\phi$ is locally constant and $U$ is open and compact.)
Then
\begin{equation}\label{Theorem 7.7.1Hormander3}
\int_U\chi(\lambda f(x))\phi(x)\,dx=0 \qquad(\lambda \in \F^\times\,,\ |\lambda|>\max\{q\delta^{-2}B, \delta^{-1}q^{m_0}\})\,.
\end{equation}
\end{lem}
\begin{prf}
With the covering as above,
\[
\int_U\chi(\lambda f(x))\phi(x)\,dx=\Sigma_k\, \chi(\lambda f(x_k))\phi(x_k)\int_{B_{q^{-m}}}\chi(\lambda f'(x_k)\cdot x)\chi(\lambda R(x_k,x)(x)\cdot x)\,dx\,.
\]
Notice that for $|\lambda|>\delta^{-1}q^m$, the function
\[
B_{q^{-m}}\ni x\to \chi(\lambda f'(x_k)\cdot x)\in \C^\times
\]
is a non-trivial character of the additive group $B_{q^{-m}}$. Indeed if this character was trivial then we would have
\[
|\lambda f'(x_k)\cdot x|\leq 1 \qquad (x\in B_{q^{-m}})\,.
\]
Choose $x\in B_{q^{-m}}$ so that $|f'(x_k)\cdot x|=|f'(x_k)||x|$. If, in addition $|x|=q^{-m}$, then 
\[
|\lambda||f'(x_k)|q^{-m}\leq 1\,.
\]
This would imply
\[
|\lambda|\delta q^{-m}\leq 1\,,
\]
which would contradict the choice of $|\lambda|>\delta^{-1}q^m$. 
By taking $m=m_0$ we see that if $B=0$ then
\begin{equation}\label{Theorem 7.7.1Hormander1}
\int_U\chi(\lambda f(x))\phi(x)\,dx=0 \qquad(\lambda \in \F^\times\,,\ |\lambda|>\delta^{-1}q^{m_0})\,.
\end{equation}
Assume from now on that $B>0$. If $|\lambda|Bq^{-2m}\leq 1$, then
\[
\chi(\lambda R(x_k,x)(x)\cdot x)=1\qquad (x\in B_{q^{-m}})\,.
\]
Hence
\begin{equation}\label{Theorem 7.7.1Hormander2}
\int_{B_{q^{-m}}}\chi(\lambda f'(x_k)\cdot x)\chi(\lambda R(x_k,x)(x)\cdot x)\,dx=0 \qquad (\delta^{-1}q^m<
|\lambda|\leq B^{-1}q^{2m})\,.
\end{equation}
Let $\delta=q^{-m_\delta}$ and $B=q^{m_B}$. Then the interval $(\delta^{-1}q^m, B^{-1}q^{2m}]=(q^{m_\delta+m}, q^{-m_B+2m}]$ is not empty if $m_\delta+m_B<m$. 
By writing $m_\delta+m_B+k=m$, we see that

\begin{eqnarray*}
&&\bigcup_{m_\delta+m_B<m}(m_\delta+m, -m_B+2m]=\bigcup_{k=1}^\infty(2m_\delta+m_B+k, 2m_\delta+m_B+2k]\\
&=&2m_\delta+m_B+\bigcup_{k=1}^\infty(k, 2k]=2m_\delta+m_B+(1,\infty)=(2m_\delta+m_B+1, \infty)\,.
\end{eqnarray*}

Since, $q^{2m_\delta+m_B+1}=q\delta^{-2}B$, we see that for any $\lambda$ with $|\lambda|> q\delta^{-2}B$ there is an $m>m_\delta+m_B$ such that $\delta^{-1}q^m<|\lambda|\leq B^{-1}q^{2m}$, so that \eqref{Theorem 7.7.1Hormander2} holds for this $m$. If in addition $m\geq m_0$, then the function $\phi$ is constant on each $x_k+B_{q^{-m}}$. Hence
\[
\int_U\chi(\lambda f(x))\phi(x)\,dx=0 \qquad(\lambda \in \F^\times\,,\ |\lambda|> q\delta^{-2}B, m\geq m_0)\,.
\]
Next we wold like to express the condition $m\geq m_0$ in terms of $|\lambda|$. If $m_0\leq m_\delta+m_B$, then $m>m_0$ because $m>m_\delta+m_B$. If $m_0> m_\delta+m_B$ then we need a condition on $|\lambda|$ such that the non-empty interval $(m_\delta+m_0, -m_B+2m_0]$ is to the left of 
$(m_\delta+m, -m_B+2m]$, with a possible overlap. This will happen if $m_\delta+m_0<\log_q|\lambda|$. Equivalently if $\delta^{-1}q^{m_0}<|\lambda|$. Hence 
\begin{equation}\label{Theorem 7.7.1Hormander4}
\int_U\chi(\lambda f(x))\phi(x)\,dx=0 \qquad(\lambda \in \F^\times\,,\ |\lambda|> 
\max\{q\delta^{-2}B,\delta^{-1}q^{m_0}\})\,.
\end{equation}
Clearly \eqref{Theorem 7.7.1Hormander3} follows from \eqref{Theorem 7.7.1Hormander1} and \eqref{Theorem 7.7.1Hormander4}.
\end{prf}
Let $X\subseteq \F^n$ be an open set. Denote by $\Ss(X)\subseteq\Ss(\F^n)$ the subset of functions supported in $X$ and let $\Ss(X)^*$ be the dual of $\Ss(X)$ with the topology of pointwise convergence (weak topology).
\begin{lem} \label{Theorem218Hormander}
Let $u_j\in \Ss(X)^*$ be a sequence such that for every $\phi\in\Ss(X)$, the set $\{u_j(\phi);\ j=1,2,3,\ldots \}$ is bounded. Then for any $\psi\in \Ss(X)$ and any $0<r<\infty$,
\[
\sup_{|\eta|\leq r}|\cF(\psi u_j)(\eta)|<\infty\,.
\]
\end{lem}
\begin{proof}
Let $\chi_{\eta}(x)= \chi(\eta\cdot x)$. Explicitly
\[
\cF(\psi u_j)(\eta)=\int_X \chi(-\eta\cdot x)\psi(x) u_j(x)\,dx  = u_j(\chi_{-\eta} \psi) \,.
\]
We claim that
\[
\chi(-\eta\cdot x)=\chi(-\eta\cdot y)\ \ \ \text{if}\ \ \ |x-y|\leq r^{-1} \,.
\]
Indeed $\chi(-\eta\cdot x)=\chi(-\eta\cdot y)$ if $|\eta \cdot (x-y)| < 1$. By the Cauchy-Schwarz inequality, $|\eta \cdot (x-y)| \leq |\eta| |x-y| \leq r |x-y|$. The claim follows.

The set of functions
\[
\{ \chi(- \eta \cdot x) : |\eta|\leq r \}
\]
is uniformly locally constant. Therefore the functions
\[
\chi(-\eta\cdot x)\psi(x) \qquad (|\eta|\leq r)
\]
are uniformly locally constant and all supported in $\supp \psi$. Hence they all belong to a finite dimensional subspace $\Ss(X)_\phi\subseteq \Ss(X)$. The restriction of the sequence $u_j$ to $\Ss(X)_\phi$ is weakly bounded.  By the Uniform Boundedness Principle, or jst because the space $\Ss(X)_\phi$ is finite dimensional, it is uniformly bounded and we are done.
\end{proof}
A subset $\Gamma\subseteq X\times(\F^n\setminus 0)$  is called a cone if $(x,\xi)\in \Gamma \iff (x,\lambda \xi)\in \Gamma$ for any $\lambda\in\Lambda$. For any closed cone $\Gamma\subseteq X\times(\F^n\setminus 0)$ let $\Ss^*_\Gamma(X)\subseteq \Ss^*(X)$ denote the subset of all the distributions whose wave front set is contained in $\Gamma$. We introduce a topology in $\Ss^*_\Gamma(X)$ as follows. Unfortunately this definition is missing in \cite{Heifetz}. This fact was noticed and corrected in 
\cite{UD2}. 
We provide a concise argument, as close as possible to the one used by H\"ormander in the real case, below.
\begin{defi}\label{topologyinssgamma}
A sequence $u_j\in \Ss^*_\Gamma(X)$ converges to $u\in \Ss^*_\Gamma(X)$ if and only if for any $\phi\in \Ss(X)$,
\begin{equation}\label{topologyinssgamma1}
\underset{j\to\infty}{\lim}\,u_j(\phi)=u(\phi)
\end{equation}
and for any open cone $V\subseteq \F^n\setminus 0$,
\begin{equation}\label{topologyinssgamma2}
\text{if}\ \ \ \Gamma\cap(\supp\phi \times V)=\emptyset\,,\ \ \ \text{then}\ \ \ 
\bigcup_j\supp\cF (\phi u_j)\cap V\ \ \ \text{is bounded}\,.
\end{equation}
\end{defi}
\begin{lem}\label{Hormender 8.2.3}\cite[Theorem 8.2.3]{Hormander}
For any $u\in \Ss_\Gamma(X)$ there is a sequence $u_j\in \Ss(X)$ such that $u=\underset{j\to\infty}{\lim} u_j$ in $\Ss^*_\Gamma(X)$.
\end{lem}
\begin{prf}
Let $u_j=\phi_j*(\psi_j u)$, where 
\begin{enumerate}[(a)]
\item $\psi_j\in\Ss(X)$ with $\psi_j=1$ on any 
compact subset of $X$ for large $j$, and 
\item $0\leq \phi_j\in\Ss(X)$, $\int\phi_j(x)\,dx=1$, $\supp \phi_j+\supp \psi_j\subseteq X$ and $\supp \phi_j$ shrinks to zero if $j$ goes to infinity.
\end{enumerate}
Then $u_j\in \Ss(X)$ and \eqref{topologyinssgamma1} holds. 
If $\phi$ and $V$ satisfy \eqref{topologyinssgamma2} choose $\psi\in \Ss(X)$ equal to $1$ in a neighborhood of $\supp \phi$ and  a closed cone $W\subseteq \F^n\setminus 0$ containing $V$ in the interior such that
\begin{equation}\label{Hormender 8.2.3.0}
\Gamma\cap(\supp\psi \times W)=\emptyset\,.
\end{equation}
Set $w_j=\phi_j*(\psi u)$. Then $\phi u=\phi \psi u$, so that $\phi u_j=\phi w_j$ for large $j$. Also
\[
|\cF  w_j|\leq |\cF \phi_j||\cF(\psi u)|\leq |\cF(\psi u)|\,.
\]
Hence,
\[
\bigcup_{j\ \text{large}}\supp \cF(\phi u_j)=\bigcup_{j\ \text{large}}\supp \cF(\phi w_j)
\subseteq \supp \cF(\psi u)\,.
\]
Corollary \ref{WFisClosed} shows that $\WF(\psi u)\subseteq \WF(u)\cap (\supp\,\psi\times \F^n)$. Hence, $\WF(\psi u)\subseteq \Gamma\cap (\supp\,\psi\times \F^n)$. 
Corollary \ref{projectionWF} implies that the projection of $\WF(\psi u)$ onto the second variable is equal to $\Sigma(\psi u)$.
Therefore the condition \eqref{Hormender 8.2.3.0} implies that the set $\supp \cF(\psi u)\cap W$ is bounded.
Hence the claim follows. 
\end{prf}
By taking $\psi_j=1$ in the proof of Lemma \ref{Hormender 8.2.3} we obtain the following Corollary.
\begin{cor}\label{Hormender 8.2.3.cor}
Fix $u\in \Ss^*_\Gamma(X)$.
Let $0\leq \phi_j\in\Ss(X)$, ${\displaystyle \int\phi_j(x)\,dx=1}$, $\supp \phi_j+\supp u\subseteq X$ and $\supp \phi_j$ shrinks to 0 if $j$ goes to infinity. Then
$u=\underset{j\to\infty}{\lim}\, \phi_j*u$ in $\Ss^*_\Gamma(X)$.
\end{cor}
\begin{thm}\label{Theorem 8.2.4}\cite[Theorem 8.2.4]{Hormander}, \cite[Theorem 2.8]{Heifetz}
Let $X\subseteq \F^m$ and $Y\subseteq \F^n$ be open subsets and let $f:X\to Y$ be an analytic map.
Set
\[
N_f=\{(f(x),\eta)\in Y\times(\F^n\setminus 0);\ f'(x)^t \eta=0\}\,.
\]
Then for any closed cone $\Gamma\subseteq Y\times(\F^n\setminus 0)$ such that
\[
\Gamma\cap N_f=\emptyset\,,
\]
The map $f^* \colon \C^\infty(Y) \rightarrow \C^\infty(X)$ given by
\[
f^*(u) = u\circ f
\]
extends uniquely to a continuous map
\[
f^*:\Ss^*_\Gamma(Y)\to \Ss^*_{f^*\Gamma}(X)\,,
\]
where
\[
f^*\Gamma=\{(x, f'(x)^t\eta);\ (x,\eta)\in \Gamma\}\,.
\]
In particular
\[
\WF(f^*u)\subseteq f^*\WF(u) \qquad (u\in \Ss^*_\Gamma(\F^n))\,.
\]
\end{thm}
\begin{proof}
Define $f^*u=u\circ f$ whenever $u\in C^\infty(Y)$. By Lemma \ref{Hormender 8.2.3} the theorem will be proven if we show that $f^*$ maps sequences $u_j\in C^\infty(Y)$ converging in $\Ss^*_\Gamma(Y)$ to sequences $f^*u_j\in C^\infty(X)$ converging in $\Ss^*_{f^*\Gamma}(Y)$. First we shall prove the convergence in $\Ss^*(X)$. 

If $u\in \Ss(X)$ and $\phi\in \Ss(X)$ then by Fourier inversion formula applied to $u$
\begin{eqnarray}\label{Theorem 8.2.5}
f^*u(\phi)&=&\int_X u(f(x))\phi(x)\,dx
=\int_X \int_{\F^n} \cF u(\eta) \chi(f(x)\cdot\eta)\,d\eta\,\phi(x)\,dx\nn\\
&=&\int_{\F^n} \cF u(\eta)\,I_\phi(\eta)\,d\eta\,,\nn\\
I_\phi(\eta)&=&\int_X \chi(f(x)\cdot\eta) \phi(x) \,dx\,.
\end{eqnarray}
The formula \eqref{Theorem 8.2.5} expresses the pullback in terms of the Fourier transform and therefore allows us to use assumptions expressed in terms of the wave front set.

Let $x_0\in X$, $y_0=f(x_0)$, $\Gamma_{y_0}=\{\eta;\ (y_0,\eta)\in\Gamma\}$.  Choose
\begin{enumerate}[(a)]
\item a closed conic neighborhood $V\subseteq \F^n\setminus 0$ of $\Gamma_{y_0}$ such that 
\[
f'(x_0)^t\eta\ne 0 \qquad (\eta \in V)
\]
(this is possible because $N_f\cap\Gamma=\emptyset$),
\item an open compact neighborhood $Y_0$ of $y_0$ such that $V$ is a neighborhood of $\Gamma_y$ for every $y\in Y_0$ (this is possible because $\Gamma$ is closed) and
\item an open compact neighborhood $X_0$ of $x_0$ with $f(X_0)\subseteq Y_0$ and 
\[
f'(x)^t\eta\ne 0 \qquad (x\in X_0\,,\ \eta \in V)\,.
\]
\end{enumerate}
Choose $\psi\in \Ss(Y_0)$ equal to $1$ on $f(X_0)$. Then \eqref{Theorem 8.2.5} implies that
\begin{equation}\label{Theorem 8.2.5.1}
f^*u(\phi)=\int_{\F^n} \cF (\psi u)(\eta)\,I_\phi(\eta)\,d\eta
\qquad (\phi \in \Ss(X_0)\,,\ u\in C^\infty(Y))\,.
\end{equation}
For a fixed $\eta\in V$ consider the function
\[
f_\eta:X_0\ni x\to f(x)\cdot\eta=f(x)^t\eta\in \F\,.
\]
Then $f_\eta$ is an analytic function with $f_\eta'(x)=f'(x)^t\eta$. Hence,
\[
\delta_\eta=\min_{x\in X_0}|f_\eta'(x)|>0\,. 
\]
Therefore $f_\eta$ satisfies the condition of Lemma \ref{Theorem 7.7.1Hormander}. Hence
\[
\int_{X_0}\chi(\lambda f(x)\cdot \eta)\phi(x)\,dx=0 \qquad(\lambda \in \F^\times\,,\ |\lambda|>\max\{q\delta_\eta^{-2}B_\eta, q^{m_0+1}, \delta_\eta^{-1}q^{m_0}\})\,,
\]
where $m_0$ is minimal such that $X_0$ is the disjoint union of balls of radius $q^{-m_0}$ on each of  which $\phi$ is constant. The parameters $\delta_\eta$ and $B_\eta$ depend continuously on $\eta$. Hence, if $C=\max_{\eta\in V, |\eta|=1} \max\{q\delta_\eta^{-2}B_\eta, q^{m_0+1}, \delta_\eta^{-1}q^{m_0}\}$, then
\[
\int_{X_0}\chi(\lambda f(x)\cdot \eta)\phi(x)\,dx=0 \qquad(\lambda \in \F^\times\,,\ |\lambda|>C,\ \eta\in V,\ |\eta|=1)\,.
\]
Therefore,
\begin{equation}\label{Theorem 8.2.6}
\supp I_\phi \cap V\\ \ \text{is bounded for any $\phi\in\Ss(X_0)$}\,.
\end{equation}
Fix $\phi\in\Ss(X_0)$ and let $u_j\in C^\infty(Y)$ converge to $u$ in $\Ss^*_\Gamma(X)$.
By combining \eqref{Theorem 8.2.6} with Lemma \ref{Theorem218Hormander} we see that
\begin{equation}\label{Theorem 8.2.6.1}
\underset{j\to\infty}{\lim}\int_{V} \cF (\psi u_j)(\eta)\,I_\phi(\eta)\,d\eta
=\int_{V} \cF (\psi u)(\eta)\,I_\phi(\eta)\,d\eta\,.
\end{equation}
Let $V^c\subseteq \F^n\setminus 0$ be the complement of $V$. Then
\[
(Y_0\times V^c)\cap \Gamma =\emptyset\,.
\]
Since $u_j$ converge in $\Ss^*_\Gamma(X)$, the set
\begin{equation}\label{Theorem 8.2.6.1.1}
\bigcup \supp \cF (\psi u_j)\cap V^c
\end{equation}
is bounded. Also, by Lemma \ref{Theorem218Hormander} the functions $\cF (\psi u_j)$ are uniformly bounded on this set. Therefore
\begin{equation}\label{Theorem 8.2.6.2}
\underset{j\to\infty}{\lim}\int_{V^c} \cF (\psi u_j)(\eta)\,I_\phi(\eta)\,d\eta
=\int_{V^c} \cF (\psi u)(\eta)\,I_\phi(\eta)\,d\eta\,.
\end{equation}
We conclude from \eqref{Theorem 8.2.6.1} and \eqref{Theorem 8.2.6.2} that 
\begin{equation}\label{Theorem 8.2.6.3}
\underset{j\to\infty}{\lim}\int_{\F^n} \cF (\psi u_j)(\eta)\,I_\phi(\eta)\,d\eta
=\int_{\F^n} \cF (\psi u)(\eta)\,I_\phi(\eta)\,d\eta\,.
\end{equation}
Thus by \eqref{Theorem 8.2.5} and partition of unity, the sequence $f^*u_j$ converges to a limit in $\Ss^*(X)$ independent of the sequence chosen. We denote this limit by $f^*u\in \Ss^*(X)$:
\begin{equation}\label{Theorem 8.2.6.4}
\underset{j\to\infty}{\lim}f^*u_j(\phi)= f^*u(\phi) \qquad (\phi\in\Ss(X))\,.
\end{equation}
Now we'll show that $\underset{j\to\infty}{\lim}f^*u_j= f^*u$ in $\Ss_\Gamma^*(X)$. We keep the notation developed between \eqref{Theorem 8.2.5} and \eqref{Theorem 8.2.6.3} with $\phi\in\Ss(X_0)$. Replacing 
$\phi$ by $\chi(-\xi\cdot x)\phi(x)$ and $u$ by $u_j$ in \eqref{Theorem 8.2.5.1} gives
\begin{eqnarray}\label{Theorem 8.2.5.0}
\cF(\phi f^*u_j)(\xi)&=&\int_X \psi(f(x))u_j(f(x))\chi(-\xi\cdot x)\phi(x)\,dx\nn\\
&=&\int_X \int_{\F^n} \cF (\psi u_j)(\eta) \chi(f(x)\cdot\eta-x\cdot \xi))\,d\eta\,\phi(x)\,dx\nn\\
&=&\int_{\F^n} \cF (\psi u_j)(\eta)\,I_\phi(\eta, \xi)\,d\eta\,,\nn\\
I_\phi(\eta, \xi)&=&\int_X \chi(f(x)\cdot\eta-x\cdot \xi)) \phi(x) \,dx\,.
\end{eqnarray}
Let $W\subseteq \F^m\setminus 0$ be an open conic neighborhood of 
$f'(x_0)^t\Gamma_{y_0}=(f^*\Gamma)_{x_0}$. We may assume that $V$ and $X_0$ are chosen so that
\begin{equation}\label{Theorem 8.2.5.0.0}
f'(x)^t\eta\in W \qquad(x\in \X_0\,,\ \eta\in V) \,.
\end{equation}
For a fixed $\eta\in V$ and $\xi\notin W$ consider the function
\[
f_{\eta,\xi}:X_0\ni x\to f(x)\cdot\eta-x\cdot \xi\in \F\,.
\]
Then $f_{\eta,\xi}$ is an analytic function with $f_{\eta,\xi}'(x)=f'(x)^t\eta-\xi$. Since $X_0$ is compact and  \eqref{Theorem 8.2.5.0.0} implies that 
\[
\delta_{\eta,\xi}=\min_{x\in X_0}|f_{\eta,\xi}'(x)|>0\,. 
\]
Therefore $f_{\eta,\xi}$ satisfies the condition of Lemma \ref{Theorem 7.7.1Hormander}. Hence
\[
\int_{X_0}\chi(\lambda (f(x)\cdot \eta-x\cdot \xi))\phi(x)\,dx=0 \qquad(\lambda \in \F^\times\,,\ 
|\lambda|>\max\{q\delta_{\eta,\xi}^{-2}B_{\eta,\xi}, q^{m_0+1}, \delta_{\eta,\xi}^{-1}q^{m_0}\})\,,
\]
where $m_0$ is minimal such that $X_0$ is the disjoint union of balls of radius $q^{-m_0}$ on each of  which $\phi$ is constant. The parameters $\delta_{\eta,\xi}$ and $B_{\eta,\xi}$ depend continuously on $\eta$ and $\xi$. Hence, if $C=\max_{\eta\in V,\ \xi\notin W,\ |\eta|=1,\ |\xi|=1} \max\{q
\delta_{\eta,\xi}^{-2}B_{\eta, \xi}, q^{m_0+1}, \delta_{\eta,\xi}^{-1}q^{m_0}\}$, then
\[
\int_{X_0}\chi(\lambda (f(x)\cdot \eta-x\cdot \xi))\phi(x)\,dx=0 \qquad(\lambda \in \F^\times\,,\ |\lambda|>C,\ \eta\in V,\ \xi\notin W,\ |\eta|=1,\ |\xi|=1)\,.
\]
Therefore,
\begin{equation}\label{Theorem 8.2.6'}
\supp I_\phi \cap (V\times W^c)\ \ \ \text{is bounded for any $\phi\in\Ss(X_0)$}\,.
\end{equation}
Let 
\[
I_{\phi, \eta}(\xi)=I_\phi(\eta, \xi)\,.
\]
By combining \eqref{Theorem 8.2.6'} with \eqref{Theorem 8.2.6.1.1} we see that
\begin{equation}\label{Theorem 8.2.6'''}
\left(\bigcup _{j}\supp \int_{V} \cF(\psi u_j)(\eta) I_{\phi, \eta}\,d\eta \right)\cap W^c\ \ \ \text{is bounded for any $\phi\in\Ss(X_0)$}\,.
\end{equation}
Let us look at the integral $I_\phi(\eta, \xi)$ as follows
\[
I_\phi(\eta, \xi)=\int_X \chi(-x\cdot \xi) \big(\chi(f(x)\cdot\eta) \phi(x)\big) \,dx\,,
\]
so that $\chi(f(x)\cdot\eta) \phi(x)$ is the test function. 
Since, for $\eta$ in a compact set, the functions $\chi(f(x)\cdot\eta) \phi(x)$ are uniformly locally constant,
Lemma \ref{Theorem 7.7.1Hormander} implies that
\[
\bigcup_{\eta\ \text{in a comact set}}\supp I_{\phi, \eta}
\]
is bounded. By combining this with \eqref{Theorem 8.2.6.1.1} we see that
\begin{equation}\label{Theorem 8.2.6''}
\bigcup _{j}\supp \int_{V^c} \cF(\psi u_j)(\eta) I_{\phi, \eta}\,d\eta \ \ \ \text{is bounded for any $\phi\in\Ss(X_0)$}\,.
\end{equation}
The statements \eqref{Theorem 8.2.6'''} and \eqref{Theorem 8.2.6''} imply
\[
\left(\bigcup _{j}\supp \int_{F^n} \cF(\psi u_j)(\eta) I_{\phi, \eta}\,d\eta \right)\cap W^c\ \ \ \text{is bounded for any $\phi\in\Ss(X_0)$}\,.
\]
But as we see from \eqref{Theorem 8.2.5.0}, this means that
\begin{equation}\label{Theorem 8.2.6''''}
\left(\bigcup _{j}\supp \cF(\phi f^* u_j)\right)\cap W^c\ \ \ \text{is bounded for any $\phi\in\Ss(X_0)$}\,.
\end{equation}
By partition of unity, this shows that $f^*u_j$ converges to $f^*u$ in $\Ss_{f^*\Gamma}(X)$.  Hence the map
\[
f^*:\Ss_\Gamma(Y)\to \Ss_{f^*\Gamma}(X)
\]
is well defined and continuous.
\end{proof}
This theorem shows that it makes sense to define the wave front set $\WF(u)\subseteq T^*X\setminus 0$ for any distribution $u\in C_c^\infty(X)^*$ on an analytic manifold $X$ by using local charts. See  \cite[page 265]{Hormander} and \cite[page 290]{Heifetz}. In particular Theorem \ref{Theorem 8.1.5Hormander} may be rewritten as

\begin{thm}\label{Theorem 8.1.5Hormander1}
Let $X$ be an analytic manifold and let $Y \subseteq X$ be a submanifold equipped with a measure $\mu$ supported on $Y$ equal to a multiple by a non-negative function in any coordinate patch of $Y$ to the standard Haar measure. We view $\mu$ as a distribution on~$X$.\footnote{ More precisely, we define $\mu(\psi) = \mu(\psi|_Y)$ for $\psi \in C_c^\infty(X)$.} Then
\[
\WF(\mu)=T^*_{{\mathrm{supp}} \, \mu} X\,,
\]
where the right hand side denotes the bundle conormal of $Y$ in $X$ supported on $\supp \mu$.
\end{thm}
This is example 8.2.5 in \cite{Hormander}.

\biblio
\end{document}